\documentclass[11pt]{amsart}

\usepackage[T1]{fontenc}
\usepackage[utf8]{inputenc}
\usepackage[english]{babel}

\usepackage{tikz}
\usepackage{lineno}

\usepackage{enumitem}
\setlist*[enumerate,1]{label=(\arabic*),itemsep=0pt}
\setlist*{nosep,topsep=1ex}

\usepackage[scaled=.75]{beramono} 
\usepackage[cal=boondoxo]{mathalfa} 

\usepackage{amsmath}
\usepackage{amsfonts}
\usepackage{amssymb}
\usepackage{upgreek}
\usepackage{xfrac} 

\usepackage{amsthm}
\usepackage[hidelinks]{hyperref}
\usepackage[capitalise,nameinlink]{cleveref}
\usepackage{xypic}
\usepackage{xcolor}

{\theoremstyle{plain}
\newtheorem{theorem}{Theorem}[section]
\newtheorem{theoremC}[theorem]{Theorem${}^\star$}
\newtheorem{proposition}[theorem]{Proposition}
\newtheorem{propositionC}[theorem]{Proposition${}^\star$}
\newtheorem{corollary}[theorem]{Corollary}
\newtheorem{corollaryC}[theorem]{Corollary${}^\star$}
\newtheorem{lemma}[theorem]{Lemma}
\newtheorem{lemmaC}[theorem]{Lemma${}^\star$}

}

{\theoremstyle{definition}
\newtheorem{definition}[theorem]{Definition}
\newtheorem{example}[theorem]{Example}
}

\crefname{theoremC}{Theorem}{Theorems}
\crefname{propositionC}{Proposition}{Propositions}
\crefname{corollaryC}{Corollary}{Corollaries}
\crefname{lemmaC}{Lemma}{Lemmas}
\crefname{figure}{Figure}{Figures}

\AddToHook{env/proposition/begin}{\crefalias{theorem}{proposition}}
\AddToHook{env/propositionC/begin}{\crefalias{theorem}{proposition}}
\AddToHook{env/corollary/begin}{\crefalias{theorem}{corollary}}
\AddToHook{env/corollaryC/begin}{\crefalias{theorem}{corollaryC}}
\AddToHook{env/lemma/begin}{\crefalias{theorem}{lemma}}
\AddToHook{env/lemmaC/begin}{\crefalias{theorem}{lemmaC}}
\AddToHook{env/fact/begin}{\crefalias{theorem}{fact}}
\AddToHook{env/definition/begin}{\crefalias{theorem}{definition}}
\AddToHook{env/example/begin}{\crefalias{theorem}{example}}

\numberwithin{equation}{section}

\usepackage{amsaddr} 

\newcommand{\defemph}[1]{\textnormal{\emph{\textbf{#1}}}}

\newcommand{\defeq}{\mathrel{{:}{=}}} 

\newcommand{\defiff}{\mathrel{{\mathop:}{\Leftrightarrow}}}

\newcommand{\pullbackcorner}[1][ul]{\save*!/#1+1.2pc/#1:(1,-1)@^{|-}\restore}

\newcommand{\abs}[1]{\lvert #1 \rvert} 

\newcommand{\all}[1]{\forall #1 .\,}
\newcommand{\some}[1]{\exists #1 .\,}

\newcommand{\limply}{\to} 
\newcommand{\lthen}{\Rightarrow} 
\newcommand{\liff}{\Leftrightarrow} 

\newcommand{\of}{\,{:}\,} 

\newcommand{\abstr}[1]{\langle#1\rangle \,}

\newcommand{\set}[1]{\{#1\}}
\newcommand{\such}{\mid}

\newcommand{\carrier}[1]{\lvert #1 \rvert} 

\newcommand{\pair}[1]{\langle #1 \rangle} 

\newcommand{\parto}{\rightharpoonup} 

\newcommand{\pow}[1]{\mathcal{P}(#1)} 
\newcommand{\powcl}[1]{\mathcal{P}_{\neg\neg}(#1)} 

\newcommand{\id}[1][]{\mathrm{id}_{#1}} 

\renewcommand{\AA}{\mathbb{A}}  
\newcommand{\PP}{\mathbb{P}}  
\newcommand{\NN}{{\mathbb{N}}}  
\newcommand{\ZZ}{\mathbb{Z}}  
\newcommand{\QQ}{\mathbb{Q}}  
\newcommand{\II}{\mathbb{I}}  

\newcommand{\KK}[1][]{\mathbb{K}^{#1}} 

\newcommand{\RRc}{\mathsf{R}_\mathsf{c}}  
\newcommand{\RRd}{\mathsf{R}}  
\newcommand{\RRcl}{\mathsf{R}_{\neg\neg}}  
\newcommand{\RR}{\mathbb{R}}  

\newcommand{\Cut}[1]{\mathsf{cut}(#1)} 
\newcommand{\ClCut}[1]{\mathsf{cut}_{\neg\neg}(#1)} 

\newcommand{\rrep}{\mathcal{r}} 
\newcommand{\srep}{\mathcal{s}} 
\newcommand{\vrep}{\comb{v}} 
\newcommand{\wrep}{\comb{w}} 

\newcommand{\compl}[1]{#1^{\mathsf{c}}} 
\newcommand{\dcompl}[1]{#1^{\mathsf{cc}}} 

\newcommand{\one}{\mathsf{1}} 
\newcommand{\two}{\mathsf{2}} 
\newcommand{\Cantor}{\two^\NN} 

\newcommand{\finseq}[1]{{#1}^{*}} 
\newcommand{\trunc}[1]{\overline{#1}} 
\newcommand{\prefix}{\sqsubseteq} 
\newcommand{\copen}[1]{{\uparrow}#1} 

\newcommand{\ClProp}{\Omega_{\neg\neg}} 

\newcommand{\mil}{\upmu} 
\newcommand{\MM}[1]{\mathbb{M}_{#1}}  
\newcommand{\TT}[1]{\mathcal{T}_{#1}} 

\newcommand{\R}[1]{\mathtt{#1}} 

\newcommand{\comb}[1]{\R{#1}} 
\newcommand{\combK}{\comb{k}} 
\newcommand{\combS}{\comb{s}} 

\newcommand{\numeral}[1]{\overline{#1}} 

\newcommand{\combCur}{\comb{cur}} 
\newcommand{\combNum}{\comb{num}} 

\newcommand{\combFst}{\comb{fst}} 
\newcommand{\combSnd}{\comb{snd}} 
\newcommand{\combPair}{\comb{pair}} 
\newcommand{\combTrue}{\comb{true}} 
\newcommand{\combFalse}{\comb{false}} 
\newcommand{\combIf}{\comb{if}} 

\newcommand{\app}[1][]{\cdot_{#1}} 
\newcommand{\defined}[1]{#1\,{\downarrow}}
\newcommand{\kleq}{\simeq} 

\newcommand{\at}{\shortmid}

\newcommand{\rz}[1][]{\Vdash_{#1}} 
\newcommand{\Ex}[1]{\mathsf{E}_{#1}} 

\newcommand{\validates}{\vDash} 

\newcommand{\objN}{\mathsf{N}} 
\newcommand{\objZ}{\mathsf{Z}} 
\newcommand{\objQ}{\mathsf{Q}} 
\newcommand{\objI}{\mathsf{I}} 

\newcommand{\invim}[1]{#1^{*}} 

\newcommand{\pr}[2][]{\upvarphi^{#1}_{#2}}

\newcommand{\rcomp}[2]{\Upphi^{#1}_{#2}}

\newcommand{\prx}[3][]{\upvarphi^{#1}_{#2}{\restriction}_{#3}} 

\newcommand{\rat}[1]{\mathsf{q}_{#1}} 

\newcommand{\ucode}[1]{\ulcorner #1 \urcorner} 

\newbox\qqBoxA
\newdimen\qqCornerHgt
\setbox\qqBoxA=\hbox{$\ulcorner$}
\global\qqCornerHgt=\ht\qqBoxA
\newdimen\qqArgHgt
\def\ucodeRaised #1{%
    \setbox\qqBoxA=\hbox{$#1$}%
    \qqArgHgt=\ht\qqBoxA%
    \ifnum     \qqArgHgt<\qqCornerHgt \qqArgHgt=0pt%
    \else \advance \qqArgHgt by -\qqCornerHgt%
    \fi \raise\qqArgHgt\hbox{$\ulcorner$} \box\qqBoxA %
    \raise\qqArgHgt\hbox{$\urcorner$}}

\newcommand{\Set}{\mathsf{Set}} 
\newcommand{\Heyt}{\mathsf{Heyt}} 
\newcommand{\op}[1]{#1^\mathsf{op}} 

\newcommand{\PredSymbol}[1][]{\mathsf{Pred}_{#1}}
\newcommand{\Pred}[2][]{\PredSymbol[#1](#2)}

\newcommand{\PRT}[1]{\mathsf{PRT}(#1)} 
\newcommand{\PAsm}[1]{\mathsf{PAsm}(#1)} 

\newcommand{\eq}[1][]{\sim_{#1}} 

\newcommand{\hull}{\mathsf{hull}} 

\newcommand{\usquare}{[0,1]^2} 
\newcommand{\psquare}{\partial\usquare} 

\title{The countable reals}
\date{April 1, 2026}
\author{Andrej Bauer}
\address{Faculty of Mathematics and Physics, University of Ljubljana, Slovenia}
\address{Institute of Mathematics, Physics and Mechanics, Slovenia}
\email{Andrej.Bauer@andrej.com}

\author{James E.~Hanson}
\email{jameseh@iastate.edu}
\address{Iowa State University, USA}

\begin{document}

\begin{abstract}
  We construct a topos in which the Dedekind reals are countable.

  The topos arises from a new kind of realizability, which we call \emph{parameterized realizability}, based on partial combinatory algebras whose application depends on a parameter. Realizers operate uniformly with respect to a given parameter set.

  Our construction uses a sequence of reals in $[0,1]$, discovered by Joseph Miller, that is non-diagonalizable in the sense that any real which is oracle-compu\-table uniformly from representations of the sequence must already appear in it. When used as the parameter set, this yields a topos in which the non-diagonalizable sequence becomes an epimorphism onto the Dedekind reals, rendering them internally countable.

  The resulting topos is intuitionistic: it refutes both the law of excluded middle and countable choice. Nevertheless, much of analysis survives internally. The Cauchy reals are uncountable. The Hilbert cube is countable, so Brouwer's fixed-point theorem follows from Lawvere's. The intermediate value theorem and the analytic form of the lesser limited principle of omniscience hold, while the limited principle of omniscience fails. Although no real-valued map has a jump, it remains open whether all such maps are continuous.

  Finally, the closed interval $[0,1]$, being countable, can be covered by a sequence of open intervals of total length less than any $\epsilon > 0$, with no finite subcover. Yet, we show that any cover using intervals with rational endpoints must admit a finite subcover.
\end{abstract}

\maketitle

\newpage

\section{Introduction}
\label{sec:introduction}

Georg Cantor's theorem \cite{cantor74:_ueber_eigen_inbeg_zahlen} about uncountability of the real numbers is one of the most widely known results of modern mathematics.
It is taught to all students of mathematics, and features regularly in popular accounts of mathematics.
Cantor proved his theorem by the diagonalization method, whose originality is attested by persistent attacks, which are still received -- and routinely rejected -- by journal editors~\cite{hodges98:_editor_recal_some_hopel_paper}.
It is unwise to participate in such quixotic attempts at bringing down the uncountability of the reals.

Nevertheless, we shall defy Cantor by constructing a mathematical universe, a topos\footnote{A topos is a kind of category, namely a finitely complete cartesian-closed one with a subobject classifier, that serves as a model of higher-order intuitionistic logic~\cite{jim86:_introd_higher_order_categ_logic,johnstone02:_sketc_eleph}.
We expect, but did not verify, that the results of~\cite{maschio2023} will yield a model of intuitionistic material set theory containing a surjection of the natural numbers onto the reals. In particular, the parameterized partial combinatory algebra from \cref{sec:topos-with-countable} should produce a corresponding implicative algebra in a natural way.} to be precise, in which the natural numbers map surjectively onto the reals.
We must of course pay a price: the topos is intuitionistic, as it invalidates the law of excluded middle and the axiom of choice.
Let us therefore first investigate how these two principles contribute to the standard proofs of uncountability of the reals.

Say that a set $A$ is \defemph{sequence-avoiding} if for every sequence $(a_n)_n$ in~$A$ there exists $x \in A$ such that $x \neq a_n$ for all $n \in \NN$. Clearly, when this is the case there can be no surjection $\NN \to A$.
Cantor's proof~\cite{cantor74:_ueber_eigen_inbeg_zahlen} showing that the reals are sequence-avoiding uses the method of nested intervals, which goes as follows.

\begin{theorem}
  \label{thm:R-uncountable}
  The set of reals is sequence-avoiding.
\end{theorem}

\begin{proof}
  Let any sequence $(a_n)_n$ of reals be given.
  It suffices to construct a sequence of closed intervals $[x_n, y_n]_n$ such that, for all $n \geq 0$:
  \begin{enumerate}[label=\alph*)]
  \item $x_n \leq x_{n+1} < y_{n+1} \leq y_n$,
  \item\label{it:width-converge} $y_n - x_n \leq 2^{-n}$,
  \item $a_n < x_{n+1}$ or $y_{n+1} < a_n$.
  \end{enumerate}
  Indeed, the limit $\ell \defeq \lim_n x_n = \lim_n y_n$ will exist and satisfy $a_n \neq \ell$ for all $n \in \NN$, because either $a_n < x_{n+1} \leq \ell$ or $\ell \leq y_{n+1} < a_n$.

  Start with any $[x_0, y_0]$ satisfying the first two conditions, for example $[0, 1]$, and define recursively
  \begin{equation}
    \label{eq:R-uncountable}%
    [x_{n+1}, y_{n+1}] \defeq
    \begin{cases}
      [x_n, (4 x_n + y_n)/5] &\text{if $(3 x_n + 2 y_n)/5< a_n$,}\\
      [(x_n + 4 y_n)/5, y_n] &\text{if $a_n < (2 x_n + 3 y_n)/5$}.
    \end{cases}
  \end{equation}
  %
  %
  In both cases the relevant inequalities are satisfied.
\end{proof}

How precisely does \eqref{eq:R-uncountable} specify a sequence, given that the two options overlap when $(3 x_n + 2 y_n)/5 < a_n < (2 x_n + 3 y_n)/5$?
We may \emph{choose} one for each~$n$ by appealing to the axiom of dependent choice,\footnote{Dependent choice is a version of the axiom of choice in which the next choice may depend on the previous ones. The proof of \cref{thm:R-uncountable} can be improved to rely on the weaker countable choice only: choose ahead of time, for all rationals $q < r$ and $n \geq 0$, either $a_n < q$ or $r < a_n$. Then make sure $(x_n)_n$ and $(y_n)_n$ are rational sequences by picking rational $x_0$ and $y_0$, and follow the choices so made when choosing an option in~\eqref{eq:R-uncountable}.} or \emph{remove} the overlap by using the law of excluded middle, which lets us take the first option if available and the second one otherwise. When both principles are proscribed, the proof crumbles.

Mathematics without the law of excluded middle and the axiom of countable choice is \emph{intuitionistic} or \emph{constructive}.\footnote{Some schools of constructive mathematics accept the axiom of countable choice, notably Erret Bishop's~\cite{bishop67:_found_const_analy}, but we eschew it because it validates \cref{thm:R-uncountable}.} Among the available varieties~\cite{bridges87:_variet_const_mathem},
intuitionistic higher-order logic~\cite{jim86:_introd_higher_order_categ_logic} suits us because it is valid in any topos.
In many respects intuitionistic and classical mathematics are alike, but there are also important differences, which we review before proceeding.

\subsection{Countable and uncountable sets in intuitionistic mathematics}
\label{sec:countabe-set-intuit}

Classically equivalent notions may bifurcate in intuitionistic mathematics, and countability is one of them.
We say that a set $A$ is \defemph{countable} if there is a surjection\footnote{The disjoint sum $\one + A$ with the singleton set $\one = \set{\star}$ make it possible to enumerate the~$\emptyset$ with the sequence of~$\star$'s.} $\NN \to \one + A$, which for an inhabited\footnote{A set is $A$ inhabited if there exists $x \in A$, thus not empty. The converse is not available intuitionistically.} set $A$ is equivalent to existence of a surjection $\NN \to A$. A set is \defemph{uncountable} if it is not countable. A sequence-avoiding set is uncountable.

Definitions of countability in terms of injection into~$\NN$ misbehave intuitionistically,
because a subset of a countable set need not be countable. The phenomenon appears in its extreme form in the
realizability topos over infinite-time Turing machines~\cite{bauer15:_baire}. In it the set\footnote{We allow ourselves to refer to the objects of a topos as ``sets'', especially when arguing internally to a topos. After all,  intuitionistic higher-order logic is a form of structural set theory, see~\cite[Sec.~2.2]{taylor99:_pract_found_mathem} and~\cite{lawvere03:_sets_mathem}.} of binary sequences $\Cantor$, where $\two \defeq \set{0,1}$, and the reals $\RR$ are both sequence-avoiding, respectively by \cref{cor:cantor-diagonal} and \cref{thm:R-uncountable} combined with the fact that a realizability topos validates the axiom of dependent choice.
At the same time, in this topos both $\Cantor$ and~$\RR$ embed into~$\NN$!

One should not succumb to a Skolem-style paradox~\cite{skolem23:_einig_bemer_begru_mengen} by conflating external and internal notions of countability.
Every set appearing in a countable model of Zermelo-Fraenkel set theory is externally countable because the whole model is, but the model validates excluded middle and hence internal uncountability of the reals.
In the effective topos~\cite{hyland82} the reals are internally uncountable because the topos validates dependent choice, hence \cref{thm:R-uncountable} works again, but are externally countable as there are only countably many Turing-computable reals.
The moral is that we should \emph{always} consider countability internally to a topos, and so far no topos is known in which the reals are internally countable~\cite{blechschmidt18}.

Georg Cantor gave another proof by diagonalization \cite{cantor79:_ueber_punkt}, which is intuitionistically valid.
We shall mark intuitionistically valid theorems with an asterisk~${}^\star$, to keep a record of those that are applicable internally to a topos.

\begin{theoremC}[Cantor]
  \label{cor:cantor-diagonal}
  The set of binary sequences $\Cantor$ is sequence-avoiding.
\end{theoremC}

\begin{proof}
  Given any $e : \NN \to \Cantor$, the map $k \mapsto e(k)(k) + 1$ differs from $e(n)$ because $e(n)(n) \neq e(n)(n) + 1$.
\end{proof}

Beware, uncountability of $\Cantor$ and~$\RR$ are unrelated intuitionistically.
We cannot exhibit a surjection in either direction,\footnote{It is consistent with intuitionistic logic~\cite[Sec.~4.6]{troelstra88:_const_mathem} that every map $\RR \to \Cantor$ is continuous, hence constant. In reverse direction, an epimorphism $\Cantor \to [0,1]$ in the topos of sheaves on~$\RR$ would have a continuous right inverse on a small enough interval, but such an inverse would have to be constant.}
and in particular we cannot show that every real has a binary, or decimal, expansion. Speaking intuitionistically, the proof of Cantor's theorem that applies diagonalization to decimal expansions of reals proves only that the set of decimal sequences is sequence-avoiding, not the reals themselves.

Likewise, observing that the powerset\footnote{Another surprise: $\Cantor$ and $\pow{\NN}$ need not be isomorphic. The former comprises the \emph{decidable} subsets of~$\NN$, which are those that either contain any given number or do not. In the absence of the law of excluded middle we cannot show that every subset of~$\NN$ is decidable.} $\pow{\NN}$ is sequence-avoiding says nothing about countability of~$\RR$.
And since a countable set may contain an uncountable subset, embedding $\Cantor$ into~$\RR$ is not helpful either.

\subsection{Real numbers in intuitionistic mathematics}
\label{sec:real-numb-intu}

The reals may be constructed in several ways~\cite[Sec.~D4.7]{johnstone02:_sketc_eleph}: the \emph{Cauchy reals} are a quotient of the set of Cauchy sequences of rationals, the \emph{Dedekind reals} are formed as Dedekind cuts of rationals, and the \emph{MacNeille reals} are the conditional order-completion of rationals.
Classically these all yield isomorphic ordered fields, but generally differ in a constructive setting.\footnote{The Cauchy and Dedekind reals coincide when the axiom of countable choice holds, see the paragraph preceding~\cite[Thm.~D4.7.12]{johnstone02:_sketc_eleph}.} Which ones should we use?

The real numbers, whatever they are, ought to form a metrically complete space. Thus we must disqualify the Cauchy reals, because without the axiom of countable choice we cannot show that a Cauchy sequence of Cauchy reals has a limit which is a Cauchy real \cite[Thm.~6.1]{lubarsky07}.

MacNeille reals are pleasingly complete: every inhabited bounded subset has a supremum.
This can be used to prove that they are sequence-avoiding without the law of excluded middle, the axiom of choice, and without diagonalization.

\begin{theoremC}
  \label{thm:fixed-point-R-uncountable}
  If every inhabited bounded set of reals has a supremum then the reals are sequence-avoiding.
\end{theoremC}

\begin{proof}
  See \cite{blechschmidt19:_knast_tarsk} for details. Briefly, given a sequence of reals $(a_n)_n$,
  the map $f : [0,2] \to [0,2]$, defined by
  \begin{equation*}
    f(x) \defeq \sup
    \set{ 2^{-n_1} + \cdots + 2^{-n_k} \such
      \max \set{a_{n_1}, a_{n_2}, \ldots, a_{n_k}} < x
    },
  \end{equation*}
  where $n_1 < n_2 < \cdots < n_k$,
  is well-defined by order-completeness and is monotone.
  If $a_n \leq f(a_n)$ then
  \begin{equation*}
    a_n + 2^{-n} \leq
    f(a_n) + 2^{-n} \leq
    f(a_n + 2^{-n}),
  \end{equation*}
  therefore if a term of the sequence is a post-fixed point then there is a post-fixed point above it.
  Consequently, the largest post-fixed point of~$f$, which exists by Tarski's fixed-point theorem~\cite[Thm.~1]{Tarski55}, cannot be a term of the sequence.
\end{proof}

Unfortunately, the MacNeille reals may fail to be even a local ring~\cite[Thm.\ D4.7.11]{johnstone02:_sketc_eleph}. We therefore adhere to the generally accepted view that the Dedekind reals are the most suitable intuitionistic notion of real numbers. We shall review their construction in detail in \cref{sec:dedek-real-numb}.

\subsection{Overview}
\label{sec:overview}

Having found no intuitionistic proofs of uncountability of the Dedekind reals, it is time to lay out a battle plan.

  Our weapon of choice against the diagonalization method will be sequences constructed by Joseph Miller~\cite{miller04:_cont_deg}. Speaking imprecisely, a sequence ${\mil : \NN \to [0,1]}$ is ``non-diagonalizable'' when any real in $[0,1]$ computed by an oracle Turing machine, uniformly in oracles representing the sequence~$\mil$, already appears in the sequence~$\mil$.
  A precise definition and Miller's construction are reviewed in \cref{sec:non-diag-sequ}.

In order to construct a topos in which a non-diagonalizable sequence appears, we devise a new kind of realizability in which logical formulas are realized uniformly in a given set of parameters. \Cref{sec:parameterized-part-comb} introduces parameterized partial combinatory algebras, which serve in \cref{sec:unif-real} to define parameterized realizability toposes using the tripos-to-topos construction. We felt it prudent to provide a generous amount of detail in these sections, even though the connoisseurs will recognize a variation on a familiar theme.

In \Cref{sec:real-numbers-object} we review the construction of Dedekind reals and formulate it in a way that makes it easy to compute the object of Dedekind reals in a parameterized realizability topos. We also calculate the object of Cauchy reals, and show that it is uncountable.

Victory is achieved in \cref{sec:topos-with-countable}. We focus on a particular parameterized realizability topos whose realizers are oracle-computable partial maps, with oracles ranging over representations of a given non-diagonalizable sequence~$\mil : \NN \to [0,1]$. We prove that $\mil$ appears in the topos as an epimorphism, from which countability of the reals follows.
We also show that the Hilbert cube is countable in the topos.

In \cref{sec:analysis-topos-tt} we take a closer look at the reals in the new topos.
First we derive the finite- and infinite-dimensional Brouwer's fixed-point theorem as trivial consequences of Lawvere's fixed-point theorem. The 1-dimensional case yields the intermediate value theorem, as well as the constructive taboo\footnote{In constructive mathematics we sometimes refer to a constructively undecided statement as a ``taboo'', especially when the statement is a consequence of excluded middle.} ${\all{x, y \in \RR} x \leq y \lor x \geq y}$. However, the topos invalidates $\all{x, y \in \RR} x = y \lor x \neq y$, as well as the taboo known as the Limited Principle of Omniscience (LPO).
We show that there are no maps $\RR \to \RR$ with a jump, but are unable to answer the more interesting question whether all such maps are continuous.
Finally, we observe that there is a countable cover of $[0,1]$ by open intervals whose lengths add up to any desired $0 < \epsilon < 1$, whence such a cover has no finite subcover. In contrast, we show that any countable cover by open intervals with rational endpoints must have a finite subcover. It is a strange topos indeed.


\section{Non-diagonalizable sequences}
\label{sec:non-diag-sequ}

As a preparation for the realizability model constructed in \cref{sec:topos-with-countable} we review the construction of non-diagonalizable sequences developed by Joseph Miller~\cite{miller04:_cont_deg}.

\subsection{Oracle-computable maps and coding of objects}
\label{sec:oracle-comp-maps}

We let lower Greek letters $\alpha$, $\beta \in \Cantor$ denote infinite binary sequences, and refer to them as \defemph{oracles}.
We write $\finseq{\two}$ for the set of finite binary sequences.
An oracle~$\alpha$ may be truncated to the first~$n$ terms to give a finite sequence $\trunc{\alpha}(n) \defeq [\alpha(0), \ldots, \alpha(n-1)]$.
Let~$\prefix$ be the prefix relation on sequences, i.e., $a \prefix b$ holds when~$b$ is an extension of~$a$, and similarly for $a \prefix \alpha$.
The sets $\copen{a} \defeq \set{\alpha \in \Cantor \such a \prefix \alpha}$ form a compact-open basis for the product topology on~$\Cantor$.

Given an oracle $\alpha \in \Cantor$, a partial map $f : \NN \parto \NN$ is \defemph{$\alpha$-computable} if it is computed by a Turing machine with access to the oracle~$\alpha$~\cite[Sec.~9.2]{rogers67:_theor_recur_funct_effec_comput}. Each such machine can be coded as a number, yielding a numbering $\pr[\alpha]{0}, \pr[\alpha]{1}, \pr[\alpha]{2}, \ldots$ of all partial $\alpha$-computable maps.
The codes describing machines are independent of the oracles. For example, there is a single index $i \in \NN$ such that $\pr[\alpha]{i} = \alpha$ for all $\alpha \in \Cantor$, and for any partial computable $f : \NN \parto \NN$ there is $j \in \NN$ such that $\pr[\alpha]{j} = f$ for all $\alpha \in \Cantor$.

Next we set up coding of mathematical objects.
Let $\pair{\Box, \Box} : \NN \times \NN \to \NN$ be a computable bijection, also known as a \defemph{pairing},
and let $\pi_1, \pi_2 : \NN \to \NN$ be the associated computable projections $\pi_1 \pair{m, n} = m$ and $\pi_2 \pair{m, n} = n$.
Let $\rat{} : \NN \to \QQ$ be a computable bijection enumerating the rationals.
Say that $f : \NN \to \NN$ \defemph{represents} $x \in \RR$ when
$\all{n \in \NN} \abs{x - \rat{f(n)}} < 2^{-n}$.
That is, $f$ enumerates (codes of) rationals that converge to~$x$ with convergence modulus~$2^{-n}$. We call such a sequence \defemph{rapidly} converging.

An oracle~$\alpha$ can be construed as the binary digit expansion of a real $\rrep(\alpha) \in [0,1]$, namely
\begin{equation*}
  \textstyle
  \rrep(\alpha) \defeq \sum_{i=0}^\infty \alpha(i) \cdot 2^{-i-1}.
\end{equation*}
The resulting map $\rrep : \Cantor \to [0,1]$ is a continuous surjection.\footnote{It is possible to arrange for~$\rrep$ to be an open quotient map, but for our purposes a continuous map that is classically surjective will do.}
From an oracle~$\alpha$ we may compute a map representing $\rrep(\alpha)$, i.e., there is $\vrep \in \NN$ such that, for all $\alpha \in \Cantor$, the map $\pr[\alpha]{\vrep}$ is total and it represents~$\rrep(\alpha)$, concretely:
\begin{equation}
  \label{eq:vrep}%
  \textstyle
  \rat{\pr[\alpha]{\vrep}(n)} = \sum_{i=0}^{n + 1} \alpha(i) \cdot 2^{-i-1}.
\end{equation}

Finally, we provide coding of sequences $\NN \to [0,1]$ with oracles. For this purpose define $\srep : \Cantor \to [0,1]^\NN$ by
\begin{equation}
  \label{eq:srep}%
  \srep(\alpha)(n) \defeq \rrep(m \mapsto \alpha(\pair{n,m}),
\end{equation}
which too is a continuous surjection.
Once again we may convert oracles to representing maps, because there is~$\wrep \in \NN$ such that $\pr[\alpha]{\R{\wrep}}(n)$ is the code of a map representing $\srep(\alpha)(n)$, concretely:
\begin{equation}
  \label{eq:wrep}%
  \pr[\alpha]{\pr[\alpha]{\R{\wrep}}(n)}(m) =
  \pr[\alpha]{\vrep}(\pair{n, m}).
\end{equation}
If $\srep(\alpha) = \srep(\beta)$ then $\pr[\alpha]{\wrep}(n)$ and $\pr[\beta]{\wrep}(n)$ both code the real~$\srep(\alpha)(n)$. In terms of \cref{def:sequence-computable} given below, $\pr[\alpha]{\wrep}(n)$ is an $\srep(\alpha)$-index of the real $\srep(\alpha)(n)$.

\subsection{Miller sequences}
\label{sec:miller-sequences}

Let us recall why in a duel between an oracle and diagonalization the latter wins.
Given $\alpha \in \Cantor$, say that~$x \in \RR$ is \defemph{$\alpha$-computable} if there is $n \in \NN$ such that $\pr[\alpha]{n}$ represents~$x$.
One might hope to construct an oracle~$\alpha \in \Cantor$ representing a sequence $a = \srep(\alpha) : \NN \to [0,1]$ that enumerates all $\alpha$-computable reals in~$[0,1]$, so that any $\alpha$-computable attempt to generate a real avoiding~$a$ would fail.
But this is not possible, because the diagonalization procedure described in the proof of~\cref{thm:R-uncountable} is itself $\alpha$-computable.
In particular, the choice in~\eqref{eq:R-uncountable} can be carried out $\alpha$-computably, one just has to compute a sufficiently precise rational approximation of~$a_n$.
The approximation, and therefore the choice and the resulting limit~$\ell$, may depend on~$\alpha$, but this in itself is not a problem.

An ingenuous insight of Joseph Miller's~\cite{miller04:_cont_deg} was that diagonalization \emph{can} be overcome, if we require oracle computations of reals to depend only on the sequence~$a$, and not the oracle~$\alpha$ representing it. In the following definition and elsewhere we write $\invim{f}$ for the inverse image map of~$f$.

\begin{definition}
  \label{def:sequence-computable}
  Given a sequence $a : \NN \to [0,1]$, say that $x \in \RR$ is \defemph{$a$-computable} if there is $n \in \NN$, called an \defemph{$a$-index}, such that $\pr[\alpha]{n}$ represents~$x$, for all $\alpha \in \invim{\srep}(a)$.
  If $n$ is an $a$-index, we define $\rcomp{a}{n}$ to be the real computed by $\pr[\alpha]{n}$, for any oracle $\alpha \in \invim{\srep}(a)$. Otherwise, $\rcomp{a}{n}$ is undefined.
\end{definition}

The diagonalization procedure from the proof of~\cref{thm:R-uncountable} is not $a$-comput\-able in the sense of the above definition, but perhaps there is a cleverer one that is? No.

\begin{theorem}[Miller]
  \label{thm:miller-sequence}%
  There exists a sequence $\mil : \NN \to [0,1]$ such that, for all $n \in \NN$, if
  $n$ is an $\mil$-index then $\mil(n) = \rcomp{\mil}{n}$.
\end{theorem}

Miller suggests ``diagonally not computably diagonalizable'' as a possible name for a sequence satisfying the stated condition. We shall call it a \defemph{Miller sequence}.
In the remainder of this section we recount the original construction~\cite[Thm.~6.3]{miller04:_cont_deg}.

\subsubsection{The interval domain}
\label{sec:interval-domain}

Let
\begin{equation*}
  \II \defeq \set{[u,v] \subseteq [0,1] \such 0 \leq u \leq v \leq 1}
\end{equation*}
be the collection of all closed sub-intervals of~$[0,1]$.
If we think of an interval $[u,v]$ as an approximate real, then it makes sense to order~$\II$ by reverse inclusion~$\supseteq$ so that the zero-width intervals~$[u,u]$ are the maximal elements.

Not every $\pr[\alpha]{n}$ represents a real, but it can be seen to represent an element of~$\II$, as follows.
For $\alpha \in \Cantor$ and $n, j \in \NN$ let the \defemph{$j$-th truncation} $\prx[\alpha]{n}{j} : \NN \to {\set{\star} \cup \NN}$ be
\begin{equation*}
  \prx[\alpha]{n}{j}(k) \defeq
  \begin{cases}
    \pr[\alpha]{n}(k) &
      \begin{aligned}[t]
        &\text{if the $n$-th machine with oracle~$\alpha$ applied} \\
        &\text{to~$k$ terminates in at most $j$ steps,}
      \end{aligned}
    \\
    \star &
    \text{otherwise.}
  \end{cases}
\end{equation*}
The oracle may be truncated too, i.e., for a finite sequence $a \in \finseq{\two}$ we define $\prx[a]{n}{j}(k)$
just like above, with $a$ acting as the oracle. If the computation accesses the oracle beyond the length of~$a$, the result is~$\star$.

Define $H^\alpha_n : \NN \to \II$ by
\begin{equation*}
  H^\alpha_n(\pair{j, k}) \defeq
  \begin{cases}
    [\rat{m} - 2^{-k}, \rat{m} + 2^{-k}] &
      \text{if $\prx[\alpha]{n}{j}(k) = m$,} \\
    [0,1] & \text{if $\prx[\alpha]{n}{j}(k) = \star$}
  \end{cases}
\end{equation*}
and $I^\alpha_n : \NN \to \II$ by $I^\alpha_n(0) \defeq [0,1]$ and
\begin{equation*}
  I^\alpha_n(k+1) \defeq
  \begin{cases}
    I^\alpha_n(k) \cap H^\alpha_n(k) &
      \text{if $I^\alpha_n(k) \cap H^\alpha_n(k) \neq \emptyset$}
    \\
    I^\alpha_n(k) & \text{otherwise.}
  \end{cases}
\end{equation*}
For any given~$n$ and~$k$, the endpoints of $I^\alpha_n(k)$ depend only on a finite prefix of~$\alpha$. Thus the endpoints depend continuously in parameter~$\alpha$ with respect to the product topology on $[0,1]^\NN$ and the discrete topology on~$\QQ$.

We get a nested sequence of closed intervals
\begin{equation*}
  [0,1] = I^\alpha_n(0) \supseteq I^\alpha_n(1) \supseteq I^\alpha_n(2) \supseteq \cdots
\end{equation*}
whose intersection is a closed interval $\mathbf{I}^\alpha_n \defeq \bigcap_{k \in \NN} I^\alpha_n(k)$.
The endpoints of~$\mathbf{I}^\alpha_n$ are $\alpha$-computable as \emph{lower} and \emph{upper} reals, respectively. Indeed, we can $\alpha$-computably enumerate a non-decreasing sequence of rationals whose supremum is the left-end point, and a non-increasing sequence of rationals whose infimum is the right-end point of~$\mathbf{I}^\alpha_n$.
Moreover, $\mathbf{I}^\alpha_n = [x,x]$ when $\pr[\alpha]{n}$ represents $x \in [0,1]$.

The story now repeats at the level of sequences. Given $a : \NN \to [0,1]$ and $n \in \NN$, for each~$\alpha$ such that $\alpha \in \invim{\srep}(a)$ the map $\pr[\alpha]{n}$ computes an interval~$\mathbf{I}^\alpha_n$ that depends on~$\alpha$, but we seek one that depends on~$a$ only.
The convex hull of the~$\mathbf{I}^\alpha_n$'s is the smallest interval that does the job:
\begin{equation*}
  \mathbf{J}^a_n \defeq
  \textstyle
  \hull \left(
    \bigcup_{\alpha \in \invim{\srep}(a)} \mathbf{I}^\alpha_n
  \right).
\end{equation*}
This is a closed interval because the union appearing in it is closed, even compact, for it is the projection of the set
\begin{equation*}
  C \defeq \set{
    (\alpha, x) \in \Cantor \times [0,1] \such
    \alpha \in \invim{\srep}(a) \land x \in \mathbf{I}^\alpha_n
  },
\end{equation*}
which we claim to be compact.
It suffices to check that $C$ is closed. Its membership relation is
\begin{align*}
  (\alpha, x) \in C
  &\liff
  \alpha \in \invim{\srep}(a) \land x \in \mathbf{I}^\alpha_n \\
  &\liff
  \alpha \in \invim{\srep}(a) \land \all{k \in \NN} x \in I^\alpha_n(k).
\end{align*}
This is a closed relation because~$\srep$ is continuous, and the endpoints of $I^\alpha_n(k)$ vary continuously in $\alpha$, $n$ and $k$.
Moreover, if $n$ happens to be an $a$-index for~$x \in [0,1]$, then $\mathbf{J}^a_n = [x,x]$ because $\mathbf{I}^\alpha_n = [x,x]$ for all $\alpha \in \invim{\srep}(a)$.

The endpoints of $\mathbf{J}^a_n$ are more complicated than those of $\mathbf{I}^\alpha_n$. There is an $\alpha$-computable double sequence of rationals $q_{i,j}$ such that the left endpoint is equal to $\inf_i \sup_j q_{i,j}$, and dually for the right endpoint.

\subsubsection{Construction of a Miller sequence}
\label{sec:constr-mill-sequ}

We prove \cref{thm:miller-sequence} by using the following generalization of Kakutani's fixed-point theorem, which itself is a generalization of Brouwer's fixed-point theorem. Depending on one's point of view, it is ironic or fascinating that such very classical theorems\footnote{Brouwer's fixed-point theorem has no constructive proofs, because a result of Orevkov's~\cite{orevkov63} implies that in the effective topos there is a continuous map $[0,1]^2 \to [0,1]^2$ which moves every point by a positive distance.} are used to construct an intuitionistic topos.

\begin{theorem}
  \label{thm:generalized-Brouwer}%
  If $F \subseteq [0,1]^\NN \times [0,1]^\NN$ is a closed set such that for each $a \in [0,1]^\NN$, the set $F[a] \defeq \set{b \in [0,1]^\NN \such (a, b) \in F}$ is non-empty and convex, then there is $\mil \in [0,1]^\NN$ such that $(\mil,\mil) \in F$.
\end{theorem}

The statement is not easy to credit properly; see the paragraph after \cite[Thm~6.1]{miller04:_cont_deg} for a discussion which proposes \cite{Eilenberg1946} as the earliest work implying the statement given here.

\Cref{thm:generalized-Brouwer} is a fixed-point theorem because a closed set $F \subseteq [0,1]^\NN \times [0,1]^\NN$ can be construed as the graph of an upper semicontinuous multivalued map taking each $a \in [0,1]^\NN$ to the non-empty set $F[a]$.
The sequence~$\mil$ is a fixed point in the sense that $\mil \in F[\mil]$.

In our case, we take $F$ to be essentially $\mathbf{J}$:
\begin{equation*}
  F \defeq \set{(a, b) \in [0,1]^\NN \times [0,1]^\NN \such \all{n \in \NN} b(n) \in \mathbf{J}^a_n},
\end{equation*}
or expressed as a multivalued map,
\begin{equation*}
  \textstyle
  F[a] \defeq \prod_{n \in \NN} \mathbf{J}^a_n.
\end{equation*}
%
%
Let us verify the conditions of the theorem.
Obviously, $F[a]$ is non-empty and convex for all $a \in [0,1]^\NN$.
To see that~$F$ is closed, we unravel its definition in logical form:
\begin{align*}
  (a, b) \in F
  &\liff \all{n \in \NN} b(n) \in \mathbf{J}^a_n \\
  &\liff\textstyle
    \all{n \in \NN} b(n) \in
    \hull \left(
      \bigcup \set{ \mathbf{I}^\alpha_n \such \srep(\alpha) = a }
    \right) \\
  &\liff\textstyle
    \begin{aligned}[t]
      &\all{n \in \NN}
      \some{u, v \in [0,1]}
      u \leq b(n) \leq v \land {} \\
      &\quad
      (\some{\alpha \in \invim{\srep}(a)} u \in \mathbf{I}^\alpha_n)
      \land
      (\some{\beta \in \invim{\srep}(a)} v \in \smash{\mathbf{I}^\beta_n})
    \end{aligned}
  \\
  &\liff\textstyle
    \begin{aligned}[t]
      &\all{n \in \NN}
      \some{u, v \in [0,1]}
      u \leq b(n) \leq v \land {} \\
      &\quad
      (\some{\alpha \in \invim{\srep}(a)} \all{m \in \NN} u \in I^\alpha_n(m))
      \land {} \\
      &\quad
      (\some{\beta \in \invim{\srep}(a)} \all{m \in \NN} v \in I^\beta_n(m))
    \end{aligned}
\end{align*}
This is a closed condition: $\forall$ and $\land$ correspond to intersection, $\exists \alpha \in \invim{\srep}(a)$ to projection along the compact set $\invim{\srep}(a) \defeq \set{ \alpha \in \Cantor \such \srep(\alpha) = a}$, the relation $\leq$ is closed, projecting the $n$-th component $b(n)$ is continuous, and the endpoints of $I^\alpha_n(m)$ depend continuously on its parameters, in particular~$\alpha$.

It remains to verify that a fixed point $\mil \in F[\mil]$ is a Miller sequence. If $n \in \NN$ is a $\mil$-index of $x \in [0,1]$ then $\mil \in F[\mil]$ implies $\mil(n) \in \mathbf{J}^a_n = [x, x]$, hence $\mil(n) = x$, as required.


\section{Parameterized partial combinatory algebras}
\label{sec:parameterized-part-comb}

We seek a topos in which a Miller sequence is an epimorphism from the natural numbers to the Dedekind reals. Some sort of realizability model seems appropriate, although it cannot be an ordinary realizability topos, as those validate the axiom of countable choice.
\Cref{def:sequence-computable} specifies that $n \in \NN$ realizes $x \in [0,1]$ when it does so \emph{parameterically} in oracles representing $a : \NN \to [0,1]$. Therefore, in the present section we develop a general notion of parameterized computational models.

Let us take a moment to introduce notation that is commonly used in realizability theory. We already wrote $f : A \parto B$ to indicate a \defemph{partial map}, which is a map $f : A' \to B$ defined on a subset $A' \subseteq A$. For $x \in A$ we write $\defined{f(x)}$ when $f(x)$ is defined, i.e., when $x \in A'$. More generally, we write $\defined{e}$ when the expression~$e$ is well-defined, and hence so are all of its subexpressions. If $e_1$ and $e_2$ are two possibly undefined expressions, then $e_1 \kleq e_2$ means that if one is defined then so is the other and they are equal. In contrast, $e_1 = e_2$ asserts that both sides are defined and equal.

In realizability theory logical statements are witnessed by \emph{realizers}, which may be numbers, $\lambda$-terms, sequences or other data. A realizer is meant to represent computational evidence of a statement.
For instance, a realizer for $\some{x} \phi(x)$ encodes a specific $a$ for which $\phi(a)$ holds as well as a realizer for $\phi(a)$, and a realizer for $\phi \to \psi$ encodes a procedure for converting realizers for $\phi$ into realizers for $\psi$. In~\cref{sec:heyt-prealg-struct} we shall make these ideas precise.

In Stephen Kleene's original realizability interpretation of Heyting arithme\-tic~\cite{KleeneSC:intint} the realizers were numbers, whereas in a typical modern framework they are elements of a structure first defined by Solomon Feferman~\cite{feferman75}:

\begin{definition}
  \label{def:pca}%
  A \defemph{partial combinatory algebra (pca)} is given by a \defemph{carrier set~$\AA$}, and a partial \defemph{application} operation ${\app} : \AA \times \AA \parto \AA$, such that there exist \defemph{basic combinators} $\combK, \combS \in \AA$ satisfying, for all $\R{a}, \R{b}, \R{c} \in \AA$,
  \begin{align*}
    &\defined{(\combK \app \R{a})}, &
    (\combK \app \R{a}) \app \R{b} &= \R{a}, \\
    &\defined{((\combS \app \R{a}) \app \R{b})}, &
    ((\combS \app \R{a}) \app \R{b}) \app \R{c} &\kleq (\R{a} \app \R{c}) \app (\R{b} \app \R{c}).
  \end{align*}
\end{definition}

\noindent
To make notation more economical, we write application $\R{a} \app \R{b}$ as juxtaposition $\R{a} \, \R{b}$ and associate it to the left, $\R{a} \, \R{b} \, \R{c} = (\R{a} \, \R{b}) \, \R{c}$.

A non-trivial pca has much richer structure as it may seem at first sight.
For instance, we may encode in it the natural numbers and partial computable functions, as we shall for parameterized pcas in \cref{sec:progr-with-ppcas}.

\begin{example}
  \label{ex:pca-K-1}%
  Recall from \cref{sec:oracle-comp-maps} that $\pr[\alpha]{m}$ stands for the partial $\alpha$-computable map coded by~$m$. The so-called Kleene's first algebra is the pca with the carrier set $\KK_1 \defeq \NN$ and application $m \cdot n \defeq \pr{m}(n)$, where $\pr{m}$ is the $m$-th partial computable map.
  For any oracle $\alpha \in \Cantor$ the relativized version $\KK[\alpha]_1$ has the same carrier set and application $m \cdot n \defeq \pr[\alpha]{m}(n)$.
\end{example}

We refer to~\cite{oosten08:_realiz} for further examples of pcas and press on to the definition of parameterized pcas.

\begin{definition}
  \label{def:ppca}%
  A \defemph{parameterized partial combinatory algebra (ppca)} is given by
  \begin{itemize}
  \item a \defemph{carrier set} $\AA$, whose elements are called \defemph{realizers},
  \item a non-empty \defemph{parameter set} $\PP$, whose elements are called \defemph{parameters},
  \item a partial \defemph{application} operation ${\app} : \PP \times \AA \times \AA \parto \AA$,
  \end{itemize}
  such that there exist \defemph{basic combinators} $\combK, \combS \in \AA$, satisfying, for all $p, q \in \PP$ and $\R{a}, \R{b}, \R{c} \in \AA$,
  \begin{align*}
    &\combK \app[p] \R{a} = \combK \app[q] \R{a},
    &
    \combK \app[p] \R{a} \app[p] \R{b} &= \R{a},
    \\
    &\combS \app[p] \R{a} = \combS \app[q] \R{a},
    &
    \combS \app[p] \R{a} \app[p] \R{b} \app[p] \R{c} &\kleq (\R{a} \app[p] \R{c}) \app[p] (\R{b} \app[p] \R{c}),
    \\
    &\combS \app[p] \R{a} \app[p] \R{b} = \combS \app[q] \R{a} \app[q] \R{b}.
  \end{align*}
\end{definition}

These conditions express uniformity of the basic combinators with respect to parameters, a notion that we shall explicate in \cref{def:uniform}.
The equations in the left column imply $\defined{(\combK \app[p] \R{a})}$ and $\defined{(\combS \app[p] \R{a} \app[p] \R{b})}$ for all $p \in \PP$ and $\R{a}, \R{b} \in \AA$.

For better readability we continue writing application as juxtaposition and let it associate to the left,
but we need a better way of displaying the parameters, which we do by writing $p \at e$ to specify
that all applications in expression~$e$ should use parameter~$p$. We assign~$\at$ a lower precedence than to application so that $p \at e_1 \app e_2 = p \at (e_1 \app e_2)$. For example, the above equations can be written as
\begin{align*}
  &p \at \combK \, \R{a} = q \at \combK \, \R{a},
  &p \at \combK \, \R{a} \, \R{b} &= \R{a},
  \\
  &p \at \combS \, \R{a} = q \at \combS \, \R{a},
  &p \at \combS \, \R{a} \, \R{b} \, \R{c} &\kleq p \at (\R{a} \, \R{c}) \, (\R{b} \, \R{c}),
  \\
  &p \at \combS \, \R{a} \, \R{b} = q \at \combS \, \R{a} \, \R{b}.
\end{align*}
We sometimes write parentheses around $p \at e$ to improve readability, especially in equations.
A formal account of the notation $p \at e$ is given in \cref{sec:comb-compl-ppcas}.

When no confusion may arise, we take the liberty of referring to a ppca $(\AA, \PP, {\cdot})$ just by the pair $(\AA, \PP)$.

\begin{example}
  An ordinary pca may be construed as a ppca whose parameter set is a singleton.
\end{example}

\begin{example}
  \label{ex:oracle-ppca}
  The following is our main motivating example.
  Recall from \cref{sec:oracle-comp-maps} that $\pr[\alpha]{m}$ stands for the partial $\alpha$-computable map coded by~$m$.
  Let the carrier of the ppca be $\KK \defeq \NN$,
  the parameter set any non-empty set of oracles $\PP \subseteq \Cantor$,
  and application $m \app[\alpha] n \defeq \pr[\alpha]{m}(n)$.

  The combinator~$\combK$ is the code of a machine which accepts input~$n$ and outputs the code of a machine that always outputs~$n$. Such a machine does not consult the oracle, and neither does the machine that always outputs~$n$, hence $\combK \app[\alpha] n = \combK \app[\beta] n$ for all $n \in \NN$.

  To obtain~$\combS$, we apply the relativized smn theorem~\cite[Sec.~III.1.5]{soare87:_recur_enumer_sets_degrees} to first get a computable map $f : \NN \times \NN \to \NN$ such that
  \begin{equation*}
    \pr[\alpha]{f(k, m)}(n) \simeq \pr[\alpha]{\pr[\alpha]{k}(n)}(\pr[\alpha]{m}(n)).
  \end{equation*}
  We apply the theorem again to get a computable $g : \NN \to \NN$ such that
  $\pr[\alpha]{g(k)}(m) = f(k, m)$, and let $\combS \in \NN$ be such that $\pr[\alpha]{\combS} = g$.
  For all $\alpha \in \PP$ and $k, m \in \NN$ we have
  \begin{equation*}
    \combS \app[\alpha] k =
    \pr[\alpha]{\combS}(k) =
    g(k)
  \end{equation*}
  and
  \begin{equation*}
    \combS \app[\alpha] k \app[\alpha] m =
    \pr[\alpha]{\pr[\alpha]{\combS}(k)}(m) =
    \pr[\alpha]{g(k)}(m) =
    f(k, m).
  \end{equation*}
  Being computable, $g$ and $f$ do not depend on the oracle, therefore
  $\combS \app[\alpha] k = \combS \app[\beta] k$
  and
  $\combS \app[\alpha] k \app[\alpha] m = \combS \app[\beta] k \app[\beta] m$ for all $\beta \in \PP$.
  Finally, the defining equation for~$f$ guarantees that
  $\combS \app[\alpha] k \app[\alpha] m \app[\alpha] n \simeq
   (k \app[\alpha] n) \app[\alpha] (m \app[\alpha] n)$.
\end{example}

\begin{example}
  \label{ex:general-oracle-ppca}%
  The previous example generalizes to Jaap van Oosten's construction~\cite[Thm~1.7.5]{oosten08:_realiz} which from any pca $(\AA, {\cdot})$ and a partial map $\xi : \AA \parto \AA$ constructs a new pca $(\AA^\xi, {\app[\xi]})$ with~$\xi$ acting as an oracle.
  The carrier set is unchanged $\AA^\xi \defeq \AA$, while application~$\app[\xi]$ represents a dialogue between a computation in~$\AA$ and the oracle~$\xi$.
  When the construction is applied to $\KK_1$ and $\alpha \in \Cantor$, we obtain a pca that is (isomorphic to)
  the relativized pca $\KK[\alpha]_1$ from \cref{ex:pca-K-1}.

  As it turns out, the construction is uniform in~$\xi$, in the sense that $\combK_\xi, \combS_\xi \in \AA^\xi$ do not depend on~$\xi$, and neither do $\combK_\xi \app[\xi] \R{a}$, $\combS_\xi \app[\xi] \R{a}$, and $\combS_\xi \app[\xi] \R{a} \app[\xi] \R{b}$.
  The conditions for having a ppca are thus met: the carrier set is $\AA$ and the parameter set is any non-empty
  subset $\PP \subseteq \AA \parto \AA$.
\end{example}


\subsection{Combinatory completeness of ppcas}
\label{sec:comb-compl-ppcas}

Partial combinatory algebras have the so-called property of \emph{combinatory completeness}.
We formulate an analogous notion for parameterized pcas.

The set of~\defemph{(applicative) expressions in variables $x_1, \ldots, x_n$} over a ppca $(\AA, \PP)$ is defined inductively:
any variable $x_i$ is an expression, any constant $\R{a} \in \AA$ is an expression, and a formal application $e_1 \cdot e_2$ is an expression if~$e_1$ and~$e_2$ are.
We continue to write application as juxtaposition and associate it to the left.
An expression~$e$ in no variables is \defemph{closed}. For any $p \in \PP$ and a closed expression~$e$,
define $p \at e$ recursively by
\begin{align*}
  p \at \R{a} &\defeq \R{a} &&\text{if $\R{a} \in \AA$,}
  \\
  (p \at e_1 \app e_2) &\defeq (p \at e_1) \app[p] (p \at e_2)
  &&\text{if $e_1$ and $e_2$ are closed expressions.}
\end{align*}
Note that $p \at e$ may be undefined.

For a variable $x$ and an applicative expression $e$, let the \defemph{abstraction} $\abstr{x} e$ be the applicative expression defined inductively as
\begin{align*}
  \abstr{x} y &\defeq \combK \, y & &\text{if $y$ is a variable distinct from $x$} \\
  \abstr{x} x &\defeq \combS \, \combK \, \combK \\
  \abstr{x} \R{a} &\defeq \combK \, a & &\text{for a constant $\R{a} \in \AA$} \\
  \abstr{x} e_1 e_2 &\defeq \combS \, (\abstr{x} e_1) \, (\abstr{x} e_2).
\end{align*}
We write $e[\R{a}_1/x_1, \ldots, \R{a}_n/x_n]$ for $e$ with $\R{a}_i$'s substituted for~$x_i$'s. We abbreviate the substitution $[\R{a}_1/x_1, \ldots, \R{a}_n/x_n]$ as $[\vec{\R{a}}/\vec{x}]$, which allows us to write $e[\vec{\R{a}}/\vec{x}]$. To be precise, substitution is defined as follows:
\begin{align*}
  x_i [\vec{\R{a}}/\vec{x}] &\defeq \R{a}_i \\
  y [\vec{\R{a}}/\vec{x}] &\defeq y &&\text{if $y \not\in \set{x_1, \ldots, x_n}$} \\
  \R{b} [\vec{\R{a}}/\vec{x}] &\defeq \R{b} &&\text{if $\R{b} \in \AA$} \\
  (e_1 \, e_2) [\vec{\R{a}}/\vec{x}] &\defeq (e_1[\vec{\R{a}}/\vec{x}]) \, (e_2[\vec{\R{a}}/\vec{x}]).
\end{align*}

\begin{lemma}
  \label{lem:abstr-subst-commute}%
  If $y \not\in \set{x_1, \ldots, x_n}$ then $(\abstr{y} e)[\vec{\R{a}}/\vec{x}] = \abstr{y} (e[\vec{\R{a}}/\vec{x}])$.
\end{lemma}

\begin{proof}
  A straightforward induction on the structure of~$e$.
\end{proof}

\begin{lemma}
  \label{lem:abstr-p-defined}
  For any applicative expression~$e$ in variables $x_1, \ldots, x_n, y$, the value $p \at (\abstr{y} e)[\vec{\R{a}}/\vec{x}]$ is defined for all $p \in \PP$ and $\R{a}_1, \ldots, \R{a}_n \in \AA$.
\end{lemma}

\begin{proof}
  We proceed by induction on the structure of~$e$:
  \begin{itemize}
  \item if $e = x_i$ then $p \at (\abstr{y} e)[\vec{\R{a}}/\vec{x}] = \combK \app[p] \R{a}_i$, which is defined,
  \item if $e = y$ then $p \at (\abstr{y} e)[\vec{\R{a}}/\vec{x}] = \combS \app[p] \combK \app[p] \combK$, which is defined,
  \item if $e = \R{a}$ for $\R{a} \in \AA$ then $p \at (\abstr{y} e)[\vec{\R{a}}/\vec{x}] = \combK \app[p] \R{a}$, which is defined,
  \item if $e = e_1 \app e_2$ then
    $
    p \at (\abstr{y} e)[\vec{\R{a}}/\vec{x}] =
    \combS \app[p] ((\abstr{y} e_1)[\vec{\R{a}}/\vec{x}]) \app[p] ((\abstr{y} e_2)[\vec{\R{a}}/\vec{x}])
    $,
    which is defined because by induction hypotheses both arguments of~$\combS$ are defined.
    \qedhere
  \end{itemize}
\end{proof}

\begin{lemma}
  \label{lem:abstr-compute}%
  For any applicative expression $e$ in variables $x_1, \ldots, x_n, y$, parameter $p \in \PP$, and $\R{a}_1, \ldots, \R{a}_n, \R{b} \in \AA$, 
  \begin{equation*}
    p \at ((\abstr{y} e) \, \R{b})[\vec{\R{a}}/\vec{x}] \kleq p \at e[\vec{\R{a}}/\vec{x}, \R{b}/y].
  \end{equation*}
\end{lemma}

\begin{proof}
  We proceed by induction on the structure of~$e$. If $e = x_i$ then
\begin{equation*}
  p \at ((\abstr{y} x_i) \, \R{b})[\vec{\R{a}}/\vec{x}] \kleq
  p \at \combK \, \R{a}_i \, \, \R{b} =
  p \at \R{a}_i =
  p \at x_i[\vec{\R{a}}/\vec{x}, \R{b}/y].
\end{equation*}
If $e = y$ then
\begin{equation*}
  p \at ((\abstr{y} y) \, \R{b})[\vec{\R{a}}/\vec{x}] \kleq
  p \at \combS \, \combK \, \combK \, \R{b} =
  p \at \R{b} =
  p \at y[\vec{\R{a}}/\vec{x}, \R{b}/y].
\end{equation*}
If $e = \R{a} \in \AA$ then
\begin{equation*}
  p \at ((\abstr{y} \R{a}) \, \R{b})[\vec{\R{a}}/\vec{x}] \kleq
  p \at \combK \, \R{a} \, \, \R{b} =
  p \at \R{a} =
  p \at \R{a}[\vec{\R{a}}/\vec{x}, \R{b}/y].
\end{equation*}
Finally, if $e = e_1 \, e_2$ then
\begin{align*}
  p \at ((\abstr{y} e_1 \, e_2) \, \R{b})[\vec{\R{a}}/\vec{x}]
  &\kleq
  p \at \combS \, ((\abstr{y} e_1)[\vec{\R{a}}/\vec{x}]) \, ((\abstr{y} e_2)[\vec{\R{a}}/\vec{x}]) \, \R{b} \\
  &\kleq
  p \at ((\abstr{y} e_1)[\vec{\R{a}}/\vec{x}] \, \R{b}) \, ((\abstr{y} e_2)[\vec{\R{a}}/\vec{x}] \, \R{b}) \\
  &\kleq
  p \at (e_1[\vec{\R{a}}/\vec{x},\R{b}/y]) \, (e_2[\vec{\R{a}}/\vec{x},\R{b}/y]) \\
  &\kleq
  p \at (e_1 \, e_2)[\vec{\R{a}}/\vec{x},\R{b}/y].
\end{align*}
The passage from the first to the second row is secured by \cref{lem:abstr-p-defined}, and from the third to the fourth by the induction hypotheses.
\end{proof}

Let us give a name to those applicative expressions that are independent of the parameter.

\begin{definition}
  \label{def:uniform}%
  A closed applicative expression~$e$ is \defemph{uniform} when $p \at e = q \at e$ for all $p, q \in \PP$.
  When this is the case, there is a unique $\ucode{e} \in \AA$ such that $p \at e = \ucode{e}$ for all $p \in \PP$. (Note that this uses the fact that $\PP$ is non-empty.)
\end{definition}

\Cref{def:ppca} postulates that $\combK \, \R{a}$ and $\combS \, \R{a} \, \R{b}$ are uniform for all $\R{a}, \R{b} \in \AA$. In subsequent calculations we shall frequently use the fact that $p \at \ucode{e} = p \at e$ when~$e$ is uniform.

\begin{lemma}
  \label{lem:abstr-uniform}%
  A closed abstraction $\abstr{x} e$ is uniform.
\end{lemma}

\begin{proof}
  We proceed by induction on the structure of~$e$.
  If $e$ is the variable $x$ then $\abstr{x} e = \combS \, \combK \, \combK$, which is uniform.
  If $e$ is a constant $\R{a} \in \AA$ then $\abstr{x} e = \combK \, \R{a}$, which is uniform.
  If $e = e_1 \, e_2$ then $\abstr{x} e = \combS \, (\abstr{x} e_1) \, (\abstr{x} e_2)$, which is uniform by induction hypotheses.
\end{proof}

\begin{theorem}[Combinatory completeness]
  \label{thm:combinatory-completeness}%
  For any applicative expression~$e$ over a ppca $(\AA, \PP)$ in variables~$x_1, \ldots, x_n, x_{n+1}$, there is $e^{*} \in \AA$ such that, for all $p \in \PP$ and $\R{a}_1, \ldots, \R{a}_{n+1} \in \AA$, the applicative expression $e^{*} \, \R{a}_1 \cdots \R{a}_n$ is uniform and
  \begin{equation*}
    (p \at e^{*} \, \R{a}_1 \cdots \R{a}_{n+1})
    \kleq
    (p \at e[\R{a}_1/x_1, \ldots, \R{a}_{n+1}/x_{n+1}]).
  \end{equation*}
\end{theorem}

\begin{proof}
  The usual proof for ordinary partial combinatory algebras can be mimicked.
  Let $e_{n+1} \defeq \abstr{x_{n+1}} e$ and $e_k = \abstr{x_k} e_{k+1}$ for $k = 1, \ldots, n$.
  By \cref{lem:abstr-uniform}, $e_1$ is uniform, so $e^{*} \defeq \ucode{e_1}$ is well defined,
  and we claim it is the element we are looking for.
  Given $p \in \PP$ and $\R{a}_1, \ldots, \R{a}_{n+1} \in \AA$, \cref{lem:abstr-compute,lem:abstr-subst-commute} imply
  \begin{align*}
    p \at e_1 \, \R{a}_1 \cdots \R{a}_n
    &\kleq p \at (e_2[\R{a}_1/x_1]) \, \R{a}_2 \cdots \R{a}_n \\
    &\kleq p \at (e_3[\R{a}_1/x_1, \R{a}_2/x_2]) \, \R{a}_3 \cdots \R{a}_n \\
    &\qquad \vdots \\
    &\kleq p \at (\abstr{x_{n+1}} e)[\R{a}_1/x_1, \ldots, \R{a}_n/x_n] \\
    &\kleq p \at \abstr{x_{n+1}} (e [\R{a}_1/x_1, \ldots, \R{a}_n/x_n]).
  \end{align*}
  The last row is defined by \cref{lem:abstr-p-defined} and the applicative expression in it is uniform by \cref{lem:abstr-uniform},
  therefore so is the first one.
  Finally, \cref{lem:abstr-compute} implies
  \begin{equation*}
    p \at e_1 \, \R{a}_1 \cdots \R{a}_{n+1}
    \kleq 
    p \at e[\R{a}_1/x_1, \ldots, \R{a}_{n+1}/x_{n+1}]. \qedhere
  \end{equation*}
\end{proof}

\subsection{Programming with ppcas}
\label{sec:progr-with-ppcas}

Combinatory completeness can be used to write complex programs in any ppca, just like in ordinary partial combinatory algebras~\cite[Sec.~1.3.1]{oosten08:_realiz}. For example, $\comb{id} \defeq \ucode{\abstr{x} x} = \ucode{\comb{s} \, \comb{k} \, \comb{k}}$ realizes the identity map.
More interesting are pairing, projections, booleans and the conditional:
\begin{align*}
  \combPair &\defeq \ucode{\abstr{x y z}{z\, x\, y}},
  &
  \combIf &\defeq \ucode{\abstr{x} x},
  \\
  \combFst &\defeq \ucode{\abstr{z}{z \, (\abstr{x\,y} x)}},
  &
  \combTrue &\defeq \ucode{\abstr{x\,y} x},
  \\
  \combSnd &\defeq \ucode{\abstr{z}{z \, (\abstr{x\,y} y)}},
  &
  \combFalse &\defeq \ucode{\abstr{x\,y} y}.
\end{align*}
These are all uniform by \cref{lem:abstr-uniform}. They satisfy the expected equations parameter-wise, for all $p \in \PP$ and $\R{a}, \R{b} \in \AA$:
\begin{align*}
  (p \at \combFst \, (\combPair \, \R{a} \, \R{b})) &= \R{a}, &
  (p \at \combIf \, \combTrue \, \R{a} \, \R{b}) &= \R{a}, \\
  (p \at \combSnd \, (\combPair \, \R{a} \, \R{b})) &= \R{b}, &
  (p \at \combIf \, \combFalse \, \R{a} \, \R{b}) &= \R{b}.
\end{align*}

Natural numbers are encoded as \defemph{Curry numerals}:
\begin{align*}
  \numeral{0} &\defeq \ucode{\combS\, \combK\, \combK},
  &
  \numeral{n+1} &\defeq \ucode{\combPair \, \combFalse \, \numeral{n}}
\end{align*}
Successor, predecessor and zero-testing are defined as
\begin{align}
  \comb{succ} &\defeq \ucode{\abstr{x}{\combPair \, \combFalse \,x}}, \label{eq:comb-succ} \\ 
  \comb{iszero} &\defeq \ucode{\combFst}, \notag\\
  \comb{pred} &\defeq
  \ucode{\abstr{x}{\combIf\, (\comb{iszero}\, x)\, \numeral{0}\, (\combSnd\, x)}}. \notag
\end{align}
These are again uniform and satisfy the expected equations.

To get recursive definitions going, we define the fixed-point combinators~$\comb{Y}$ and~$\comb{Z}$:
\begin{align*}
  \R{W} &\defeq \ucode{\abstr{x \, y} y \, (x \, x \, y)},
  &
  \comb{Y} &\defeq \ucode{\R{W} \, \R{W}},
  \\
  \R{X} &\defeq \ucode{\abstr{x\, y\, z} y \, (x \, x \, y) \, z},
  &
  \comb{Z} &\defeq \ucode{\R{X} \, \R{X}}.
\end{align*}
Then for all $p \in \PP$ and $\R{f}, \R{a} \in \AA$, $\comb{Z} \, \R{f}$ is uniform,
\begin{equation*}
  p \at \comb{Y} \, \R{f} \kleq p \at \R{f} \, (\comb{Y} \, \R{f})
  \qquad\text{and}\qquad
  p \at \comb{Z} \, \R{f} \, \R{a} \kleq p \at \R{f} \, (\comb{Z} \, \R{f}) \, \R{a}.
\end{equation*}
%
%
%
%
For instance, by repeatedly using \cref{lem:abstr-compute} we compute
\begin{equation*}
  p \at \comb{Z} \, \R{f} \, \R{a} \kleq
  p \at \R{X} \, \R{X} \, \R{f} \, \R{a} \kleq
  p \at \R{f} \, (\R{X} \, \R{X} \, \R{f}) \, \R{a} \kleq
  p \at \R{f} \, (\comb{Z} \, \R{f}) \, \R{a}.
\end{equation*}
With $\comb{Y}$ in hand primitive recursion on natural numbers is realized as
\begin{equation*}
  \comb{primrec} \defeq
  \ucode{\abstr{x \, \R{f} \, m} ((\comb{Z} \, \R{R}) \, x \, \R{f} \, m \, \comb{id})}
\end{equation*}
where
\begin{equation*}
  \R{R} \defeq \ucode{
      \abstr{r \, x \, \R{f} \, m}
      \comb{if} \, (\comb{iszero} \, m) \,
          (\combK \, x) \,
          (\abstr{y} \R{f} \, (\comb{pred} \, m) \, (r \, x \, \R{f} \, (\comb{pred} \, m) \, \comb{id}))
  }.
\end{equation*}
It satisfies, for all $p \in \PP$, $\R{a}, \R{f} \in \AA$ and $n \in \NN$,
\begin{align*}
  (p \at \comb{primrec} \, \R{a} \, \R{f} \, \numeral{0}) &= \R{a},
  &
  (p \at \comb{primrec} \, \R{a} \, \R{f} \, \numeral{n + 1}) &\kleq
  \R{f} \, \numeral{n} \, (\comb{primrec} \, \R{a} \, \R{f} \, \numeral{n}).
\end{align*}

%
%
%

\begin{example}
  \label{ex:numers-vs-numerals}
  It will be useful to know that in the ppca $(\KK, \PP)$ from \cref{ex:oracle-ppca} numbers can be converted to numerals and vice versa. For this purpose we construct realizers $\combNum, \combCur \in \KK$ such that for all $\alpha \in \PP$ and $n \in \NN$
  \begin{equation*}
    \alpha \at \combNum \, \numeral{n} = \alpha \at n
    \quad\text{and}\quad
    \alpha \at \combCur \, n = \alpha \at \numeral{n}.
  \end{equation*}
  To convert numerals to numbers, observe that there is $s \in \NN$, independent of~$\alpha$, such that $\pr[\alpha]{s}(n) = n + 1$, and define
  $\combNum \defeq \ucode{\comb{primrec} \, 0 \, s}$.
  To implement the reverse translation, we apply
  the relativized Kleene's recursion theorem~\cite[Sec.~III.1.6]{soare87:_recur_enumer_sets_degrees}
  to find $r \in \NN$, independent of~$\alpha$, such that
  \begin{equation*}
    \pr[\alpha]{r}(n) =
    \begin{cases}
      \ucodeRaised{\numeral{0}} & \text{if $n = 0$,}\\
      \pr[\alpha]{\ucode{\comb{succ}}}(\pr[\alpha]{r}(n-1)) & \text{if $n > 0$.}
    \end{cases}
  \end{equation*}
  We may take $\combCur \defeq r$ because
  $\alpha \at r \, 0 = \pr[\alpha]{r}(0) = \ucode{\numeral{0}} = (\alpha \at \numeral{0})$
  and, assuming $\alpha \at r \, n = \alpha \at \numeral{n}$ for the sake of the induction step,
  \begin{align*}
    \alpha \at r \, (n+1)
    &= \pr[\alpha]{r}(n+1)
     = \pr[\alpha]{\ucode{\comb{succ}}}(\pr[\alpha]{r}(n)) \\
    &= \alpha \at \comb{succ} \, (r \, n)
     = \alpha \at \comb{succ} \, \numeral{n}
     = \alpha \at \numeral{n + 1}.
  \end{align*}
\end{example}


\section{Parameterized realizability}
\label{sec:unif-real}

We next devise a notion of realizability based on ppcas that captures the uniformity of oracle computations from \cref{sec:non-diag-sequ}.
We use the tripos-to-topos construction~\cite{hyland80:_tripos}, a general technique for defining toposes. We refer the readers to~\cite[Sec.~2.1]{oosten08:_realiz} for background material.
For the remainder of this section we fix a ppca $(\AA, \PP)$.

\subsection{The parameterized realizability tripos}
\label{sec:tripos-built-from}

In the first step of the construction we shall define a contravariant functor
\begin{equation*}
  \PredSymbol : \op{\Set} \to \Heyt,
\end{equation*}
from sets to Heyting prealgebras, satisfying further conditions to be given later.
%
Recall that a Heyting prealgebra $(H, {\leq})$ is a set $H$ with a reflexive transitive relation~$\leq$ with elements $\bot$, $\top$ and binary operations $\land$, $\lor$, $\limply$, satisfying the laws of intuitionistic propositional calculus.

For any set~$X$, define
\begin{equation*}
  \Pred[\AA,\PP]{X} \defeq (\pow{\AA}^X, {\leq_X}),
\end{equation*}
where the preorder~$\leq_X$ will be defined momentarily.
When no confusion can arise, we abbreviate $\Pred[\AA,\PP]{X}$ as $\Pred{X}$.

An element $\phi \in \Pred{X}$ is called a \defemph{(tripos) predicate} on~$X$.
We say that $\R{a} \in \phi(x)$ \defemph{realizes} $\phi(x)$.
For a closed applicative expression~$e$ over~$\AA$, we define $e \rz[p] \phi(x)$ by
\begin{equation*}
  e \rz[p] \phi(x)
  \defiff
  \defined{(p \at e)} \land (p \at e) \in \phi(x)
\end{equation*}
and read it as ``$e$ realizes $\phi(x)$ at $p$''.
The preoder on $\Pred{X}$ is defined as follows, where $\phi, \psi \in \Pred{X}$:
\begin{equation*}
  \phi \leq_X \psi
  \defiff
    \some{\R{a} \in \AA}
    \all{x \in X}
    \all{\R{b} \in \phi(x)}
    \all{p \in \PP}
    \R{a} \, \R{b} \rz[p] \psi(x).
\end{equation*}
We say that the existential witness $a$ in the above formula \defemph{realizes} $\phi \leq_X \psi$.
Reflexivity of~$\leq_X$ is realized by $\ucode{\abstr{x} x}$. For transitivity, one checks that if $\R{a}$ realizes $\phi \leq_X \psi$ and $\R{b}$ realizes $\psi \leq_X \chi$ then $\ucode{\abstr{x} \R{b} \, (\R{a} \, x)}$ realizes $\phi \leq_X \chi$.

In order for $\PredSymbol$ to be a bona fide functor, we let it take a map $r : Y \to X$ to the \defemph{reindexing} map $\invim{r} : \Pred{X} \to \Pred{Y}$, which acts by precomposition $\invim{r} \phi = \phi \circ r$. This is obviously contravariant and functorial, and we shall check that $\invim{r}$ is a homomorphism below.

In order to show that $\PredSymbol$ is a tripos, we must verify the following conditions:
\begin{enumerate}
\item For every set $X$ the poset $\Pred{X}$ is a Heyting prealgebra (\cref{sec:heyt-prealg-struct}).
\item Reindexing is a homomorphism of Heyting prealgebras (\cref{sec:monot-reind}).
\item Universal and existential quantifiers exist for $\Pred{X}$ (\cref{sec:quantifiers}).
\item There is a generic element (\cref{sec:generic-element}).
\end{enumerate}
The arguments closely parallel those for the tripos arising from an ordinary pca~\cite[Prop.~1.2.1]{oosten08:_realiz}, the only additional care required concerns the presence of parameters.
In particular, when the parameter set~$\PP$ is a singleton, the tripos defined here coincides with the realizability tripos, and the topos defined in \cref{sec:unif-real-topos} coincides with the realizability topos.

\subsection{The Heyting prealgebra structure}
\label{sec:heyt-prealg-struct}

The Heyting structure on $\Pred{X}$ is as follows:
{\allowdisplaybreaks
\begin{align*}
  \top(x) &\defeq \AA,\\
  \bot(x) &\defeq \emptyset,\\
  (\phi \land \psi)(x) &\defeq \set{\R{a} \in \AA \such
      \all{p \in \PP}
      \combFst \, \R{a} \rz[p] \phi(x) \land
      \combSnd \, \R{a} \rz[p] \psi(x)
  },
  \\
  (\phi \lor \psi)(x) &\defeq \{\R{a} \in \AA \such
    \all{p \in \PP}
    \begin{aligned}[t]
    &(p \at \combFst \, \R{a} = \ucode{\combTrue} \land \combSnd \, \R{a} \rz[p] \phi(x))
    \lor {} \\
    &(p \at \combFst \, \R{a} = \ucode{\combFalse} \land \combSnd \, \R{a} \rz[p] \psi(x)) \},
    \end{aligned}
  \\
  (\phi \limply \psi)(x) &\defeq \set{\R{a} \in \AA \such
    \all{p \in \PP} \all{\R{b} \in \phi(x)} \R{a} \, \R{b} \rz[p] \psi(x)}.
\end{align*}
}
The above is like the analogous Heyting structure for ordinary pcas, except that realizers must be uniform in~$p \in \PP$. Next, we verify that the given operations satisfy the laws of a Heyting prealgebra.

\subsubsection{Falsehood and truth}
\label{sec:falsehood-truth}

Both $\bot \leq_X \phi$ and $\phi \leq_X \top$ are realized by $\ucode{\combK \, \combK}$.

\subsubsection{Conjunction}
\label{sec:conjunction}

We need to verify that, for all $\phi, \psi, \chi \in \Pred{X}$,
\begin{equation*}
  (\chi \leq_X \phi) \land (\chi \leq_X \psi) \iff \chi \leq_X \phi \land \psi.
\end{equation*}
If $\R{a}$ and $\R{b}$ realize $\chi \leq_X \phi$ and $\chi \leq_X \psi$, respectively, then $\chi \leq_X \phi \land \psi$ is realized by $\R{c} \defeq \ucode{\abstr{u} \combPair \, (\R{a} \, u) \, (\R{b} \, u)}$. Indeed, for any $x \in X$, $p \in \PP$ and $\R{d} \in \chi(x)$, we have
\begin{equation*}
  p \at \combFst \, (\R{c} \, \R{d})
  \kleq
  p \at \combFst \, (\combPair \, (\R{a} \, \R{d}) \, (\R{b} \, \R{d}))
  \kleq
  p \at \R{a} \, \R{d}.
\end{equation*}
Because $\R{a} \, \R{d} \rz[p] \phi(x)$, it follows that $\combFst \, (\R{c} \, \R{d}) \rz[p] \phi(x)$.
The argument for the second component is analogous.

Conversely, if $\R{a}$ realizes $\chi \leq_X \phi \land \psi$ then $\R{b} \defeq \ucode{\abstr{u} \combFst \, (\R{a} \, u)}$ and $\R{c} \defeq \ucode{\abstr{v} \combSnd \, (\R{a} \, v)}$ realize $\chi \leq_X \phi$ and $\chi \leq_X \psi$, respectively. Indeed, for any $x \in X$, $p \in \PP$ and $\R{d} \in \chi(x)$, we have $\R{a} \, \R{d} \rz[p] (\phi \land \psi)(x)$ and $p \at \R{b} \, (\R{a} \, \R{d}) = p \at \combFst \, (\R{a} \, \R{d})$, hence $\R{b} \, (\R{a} \, \R{d}) \rz[p] \phi(x)$.
The argument for $\R{c}$ and $\chi \leq_X \psi$ is analogous.

\subsubsection{Disjunction}
\label{sec:disjunction}

Disjunction is characterized by
\begin{equation*}
  (\phi \leq_X \chi) \land (\psi \leq_X \chi) \iff \phi \lor \psi \leq_X \chi.
\end{equation*}
If $\R{a}$ and $\R{b}$ respectively realize $\phi \leq_X \chi$ and $\psi \leq_X \chi$, then $\phi \lor \psi \leq_X \chi$ is realized by
\begin{equation*}
  \R{c} \defeq
  \ucode{\abstr{u} \combIf \, (\combFst \, u) \, (\R{a} \, (\combSnd \, u)) \, (\R{b} \, (\combSnd \, u))}.
\end{equation*}
Consider any $x \in X$, $p \in \PP$ and $\R{d} \in (\phi \lor \psi)(x)$.
If $p \at \combFst \, \R{d} = \combTrue$ then
\begin{equation*}
  p \at \R{c} \, \R{d}
  \kleq
  p \at \combIf \, (\combFst \, \R{d}) \, (\R{a} \, (\combSnd \, \R{d})) \, (\R{b} \, (\combSnd \, \R{d}))
  \kleq
  p \at \R{a} \, (\combSnd \, \R{d}),
\end{equation*}
and since $\R{a} \, (\combSnd \, \R{d}) \rz[p] \chi(x)$ also $\R{c} \, \R{d} \rz[p] \chi(x)$.
If $p \at \combFst \, \R{d} = \combFalse$ then
\begin{equation*}
  p \at \R{c} \, \R{d}
  \kleq
  p \at \combIf \, (\combFst \, \R{d}) \, (\R{a} \, (\combSnd \, \R{d})) \, (\R{b} \, (\combSnd \, \R{d}))
  \kleq
  p \at \R{b} \, (\combSnd \, \R{d}),
\end{equation*}
and since $\R{b} \, (\combSnd \, \R{d}) \rz[p] \chi(x)$ also $\R{c} \, \R{d} \rz[p] \chi(x)$.

Conversely, if $\R{c}$ realizes $\phi \lor \psi \leq_X \chi$ then $\phi \leq_X \chi$ and $\psi \leq_X \chi$ are respectively realized by $\R{a} \defeq \ucode{\abstr{u} \R{c} \, (\combPair \, \combTrue \, u)}$ and $\R{b} \defeq \ucode{\abstr{v} \R{c} \, (\combPair \, \combFalse \, v)}$.
Indeed, for any $x \in X$, $p \in \PP$ and $\R{d} \in \phi(x)$ we have $\combPair \, \combTrue \, \R{d} \rz[p] (\phi \lor \psi)(x)$, hence $\R{c} \, (\combPair \, \combTrue \, \R{d}) \rz[p] \chi(x)$. Now $\R{a} \, \R{d} \rz[p] \chi(x)$ holds because
$p \at \R{a} \, \R{d} \kleq \allowbreak p \at \R{c} \, (\combPair \, \combTrue \, \R{d})$.
The argument for $\R{b}$ and $\psi \leq_X \chi$ is analogous.

\subsubsection{Implication}
\label{sec:implication}

Implication is more interesting, as it involves uniformity of realizers.
It is characterized by the adjunction
\begin{equation*}
   \phi \leq_X \psi \limply \chi\iff \phi \land \psi \leq_X \chi.
\end{equation*}
If $\R{a}$ realizes $\phi \leq_X \psi \limply \chi$ then $\R{b} \defeq \ucode{\abstr{x} \R{a} \, (\combFst \, x) \, (\combSnd \, x)}$ realizes $\phi \land \psi \leq_X \chi$. Indeed, for any $x \in X$, $p \in \PP$ and $\R{c} \in (\phi \land \psi)(x)$ we have
$
  p \at \R{b} \, \R{c}
  \kleq \allowbreak
  p \at \R{a} \, (\combFst \, \R{c}) \, (\combSnd \, \R{c})
$
and $\R{a} \, (\combFst \, \R{c}) \, (\combSnd \, \R{c}) \rz[p] \chi(x)$, hence $\R{b} \, \R{c} \rz[p] \chi(x)$.

If $\R{b}$ realizes $\phi \land \psi \leq_X \chi$, then $\R{a} \defeq \ucode{\abstr{u} \abstr{v} \R{b} \, (\combPair \, u \, v)}$ realizes $\phi \leq_X \psi \limply \chi$.
To see this, we must verify for any $x \in X$, $p \in \PP$ and $\R{c} \in \phi(x)$ that $\R{a} \, \R{c} \rz[p] (\psi \limply \chi)(x)$.
Consider any $q \in \PP$ and $\R{d} \in \psi(x)$.
By \cref{lem:abstr-uniform}
\begin{equation*}
  p \at \R{a} \, \R{c} =
  p \at \abstr{v} \R{b} \, (\combPair \, \R{c} \, v) =
  q \at \abstr{v} \R{b} \, (\combPair \, \R{c} \, v) =
  q \at \R{a} \, \R{c},
\end{equation*}
hence
$
  q \at (\R{a} \app[p] \R{c}) \, \R{d} \kleq
  q \at \R{a} \, \R{c} \, \R{d} \kleq
  q \at \R{b} \, (\combPair \, \R{c} \, \R{d})
$,
and because $\R{b} \, (\combPair \, \R{c} \, \R{d}) \rz[q] \chi(x)$, it follows that $(\R{a} \app[p] \R{c}) \, \R{d} \rz[q] \chi(x)$.

\subsubsection{Negation}
\label{sec:negation}

In intuitionistic logic negation $\neg \phi$ is defined as $\phi \limply \bot$. A short calculation reveals that
\begin{align*}
  (\neg \phi)(x) &= \set{\R{a} \in \AA \such \phi(x) = \emptyset} \\
  (\neg\neg \phi)(x) &= \set{\R{a} \in \AA \such \phi(x) \neq \emptyset}.
\end{align*}

\subsection{Reindexing preserves the Heyting structure}
\label{sec:monot-reind}

We should not forget to check that $\invim{r} : \Pred{X} \to \Pred{Y}$ induced by $r : Y \to X$ is a homomorphism of Heyting prealgebras. This is easy, one just checks directly that $\invim{r}$ commutes with the logical connectives by unfolding the definitions. For example,
\begin{equation*}
  \R{a} \in \invim{r}(\phi \limply \psi)(y)
  \iff
  \all{\R{b} \in \phi(r(y))}
  \all{p \in \PP}
  \R{a} \, \R{b} \rz[p] \psi(r(y))
\end{equation*}
and
\begin{equation*}
  \R{a} \in (\invim{r}\phi \limply \invim{r}\psi)(y)
  \iff
  \all{\R{b} \in \phi(r(y))}
  \all{p \in \PP}
  \R{a} \, \R{b} \rz[p] \psi(r(y)),
\end{equation*}
which are the same condition.

\subsection{The quantifiers}
\label{sec:quantifiers}

Let $r : Y \to X$ be a map. The universal and existential quantifiers along~$r$ are monotone maps
\begin{equation*}
  \exists_r, \forall_r : \Pred{Y} \to \Pred{X},
\end{equation*}
such that, for all $\phi \in \Pred{Y}$ and $\psi \in \Pred{X}$,
\begin{equation*}
  \exists_r \phi \leq_X \psi \iff \phi \leq_Y \invim{r} \psi
  \qquad\text{and}\qquad
  \psi \leq_X \forall_r \phi \iff \invim{r} \psi \leq_Y \psi.
\end{equation*}
(The usual quantifiers correspond to $r : X \times Y \to X$ being the first projection.)
We may take the following definition of the existential quantifier:
\begin{align*}
  (\exists_r \phi)(x) \defeq
   \set{\R{a} \in \AA \such \some{y \in Y} r(y) = x \land \R{a} \in \phi(y)}.
\end{align*}
If $\R{a}$ realizes $\exists_r \phi \leq_X \psi$ then it also realizes $\phi \leq_Y \invim{r} \psi$:
for any $y \in Y$, $p \in \PP$ and $\R{b} \in \phi(y)$ we have $\R{b} \in (\exists_r \phi)(r(y))$, therefore $\R{a} \, \R{b} \rz[p] \psi(r(y))$.
Conversely, if $\R{a}$ realizes $\phi \leq_Y \invim{r} \psi$ then it also realizes $\exists_r \phi \leq_X \psi$: for any $x \in X$, $p \in \PP$ and $\R{b} \in (\exists_r \phi)(x)$, we have $r(y) = x$ for some $y \in Y$ such that $\R{b} \in \phi(y)$, hence $\R{a} \, \R{b} \rz[p] \psi(r(y))$ and $\R{a} \, \R{b} \rz[p] \psi(x)$.

Next, the definition of the universal quantifier is
\begin{multline*}
  (\forall_r \phi)(x) \defeq
   \{\R{a} \in \AA \such
     \all{y \in Y} r(y) = x \lthen
     \all{\R{b} \in \AA} \all{q \in \PP}
     \R{a} \, \R{b} \rz[q] \phi(y)
   \}.
\end{multline*}
If~$\R{a}$ realizes $\psi \leq_X \forall_r \phi$ then $\R{b} \defeq \ucode{\abstr{x} \R{a} \, x \, \combK}$ realizes $\invim{r}\psi \leq_Y \phi$:
for any $y \in Y$, $p \in \PP$, and $\R{c} \in \psi(r(y))$, we have $\R{a} \, \R{c} \rz[p] (\forall_r \phi)(r(y))$, therefore $\R{a} \, \R{c} \, \combK \rz[p] \phi(y)$ and $p \at \R{b} \, \R{c} = p \at \R{a} \, \R{c} \, \combK$, giving the required $\R{b} \, \R{c} \rz[p] \phi(y)$.
Conversely, if $\R{b}$ realizes $\invim{r}\psi \leq_Y \phi$ then $\R{a} \defeq \ucode{\abstr{x} \abstr{z} \R{b} \, x}$ realizes $\psi \leq_X \forall_r \phi$: consider any $x \in X$, $p \in \PP$, $\R{c} \in \psi(x)$. To show $\R{a} \, \R{c} \rz[p] (\forall_r \phi)(x)$, first note that $\defined{(p \at \R{a} \, \R{c})}$ by \cref{thm:combinatory-completeness}. Suppose $y \in Y$ is such that $r(y) = x$, and consider any $\R{d} \in \AA$ and $q \in \PP$.
By \cref{lem:abstr-uniform}
\begin{equation*}
  p \at \R{a} \, \R{c} =
  p \at \abstr{z} \R{b} \, \R{c} =
  q \at \abstr{z} \R{b} \, \R{c} =
  q \at \R{a} \, \R{c}.
\end{equation*}
Therefore
$
  q \at (\R{a} \app[p] \R{c}) \, \R{d} \kleq
  q \at \R{a} \, \R{c} \, \R{d} \kleq
  q \at \R{b} \, \R{c}
$.
From $\R{c} \in \psi(r(y))$ it follows that $\R{b} \, \R{c} \rz[q] \phi(y)$, therefore $(\R{a} \app[p] \R{c}) \, \R{d} \rz[q] \phi(y)$.

The reader may have expected the following, simpler definition of the universal quantifier
\begin{equation}
  \label{eq:alternative-forall}%
  (\forall'_r \phi)(x) \defeq
   \{\R{a} \in \AA \such
     \all{y \in Y} r(y) = x \lthen
     \R{a} \in \phi(y)
   \}.
\end{equation}
It is easy to check that $\forall'_r \phi \leq_X \forall_r \phi$ is realized by~$\combK$.
The converse $\forall_r \phi \leq \forall'_r \phi$ is realized by $\R{c} \defeq \abstr{x} x \, \combK$, but only when~$r$ is surjective. To see this, consider any $x \in X$, $\R{a} \in (\forall_r \phi)(x)$ and $p \in \PP$.
First, $p \at \R{c} \, \R{a} \simeq p \at \R{a} \, \combK$  is defined because~$r$ is surjective.
Second, if $y \in Y$ and $r(y) = x$ then $\R{a} \, \combK \rz[p] \phi(y)$ and $p \at \R{c} \, \R{a} = p \at \R{a} \, \combK$, therefore $\R{a} \, \combK \in \phi(y)$.

It remains to verify the Beck-Chevalley condition~\cite[Def.~2.1.2]{oosten08:_realiz}, which states that, given a pullback in~$\Set$
\begin{equation*}
  \xymatrix{
    {Y} \pullbackcorner
    \ar[r]^{r} \ar[d]_{u} 
    &
    {X} \ar[d]^{v}
    \\
    {Z} \ar[r]_{q}
    &
    {W}
  }
\end{equation*}
$\forall_r \circ \invim{u}$ and $\invim{v} \circ \forall_q$ are isomorphic maps of preorders.
For $\phi \in \Pred{Z}$, $x \in X$
the condition $\R{a} \in ((\forall_r \circ \invim{u}) \phi)(x)$ unfolds to
\begin{equation}
  \label{eq:bc-1}%
  \all{y \in Y} r(y) = x \lthen 
    \all{\R{b} \in \AA} \all{p \in \PP}
       \R{a} \, \R{b} \rz[p] \phi(u(y))
\end{equation}
while $\R{a} \in ((\invim{v} \circ \forall_q) \phi)(x)$ unfolds to
\begin{equation}
  \label{eq:bc-2}%
  \all{z \in Z} q(z) = v(x) \lthen 
    \all{\R{b} \in \AA} \all{p \in \PP}
      \R{a} \, \R{b} \rz[p] \phi(z).
\end{equation}
Let us show that these are equivalent. Suppose $\R{a}$ satisfies \eqref{eq:bc-1} and $z \in Z$ is such that $q(z) = v(x)$. Because the square is a pullback there is a unique $y \in Y$ such that $r(y) = x$ and $u(y) = z$. By \eqref{eq:bc-1} we get $\all{\R{b} \in \AA} \all{p \in \PP} \R{a} \, \R{b} \rz[p] \phi(u(y))$, which is the
same as $\all{\R{b} \in \AA} \all{p \in \PP} \R{a} \, \R{b} \rz[p] \phi(z)$.
Conversely, if $\R{a}$ satisfies \eqref{eq:bc-2} and there is a $y \in Y$ such that $r(y) = x$, then we instantiate \eqref{eq:bc-2} with $z = u(y)$ to obtain the desired $\all{\R{b} \in \AA} \all{p \in \PP} \R{a} \, \R{b} \rz[p] \phi(u(y))$.

\subsection{The generic element}
\label{sec:generic-element}

Because $\Set$ is cartesian closed, the remaining requirement for a tripos is the existence of a generic element, see the remark following~\cite[Def.~2.1.2]{oosten08:_realiz}. Specifically, we seek a set $S$ and $\sigma \in \Pred{S}$ such that, for all $X$ and $\phi \in \Pred{X}$, there exists $r_\phi : X \to S$ for which $\phi$ and $\invim{r_\phi} \sigma$ are isomorphic.

Once again, we just reuse the generic element for a tripos based on a pca, namely $S \defeq \pow{\AA}$ and $\sigma \defeq \id[\pow{\AA}]$. This obviously works because $\invim{\phi} \id[\pow{\AA}] = \phi$ for any $\phi \in \Pred{X}$.

\subsection{Tripos logic}
\label{sec:tripos-logic}

A formula $\phi$ built from logical connectives, quantifiers, and tripos predicates
whose free variables $x_1, \ldots, x_n$ range over the sets $X_1, \ldots, X_n$,
determines a tripos predicate
\begin{equation*}
  [x_1 \of X_1, \ldots, x_n \of X_n \such \phi] : X_1 \times \cdots \times X_n \to \pow{\AA},
\end{equation*}
which we sometimes abbreviate as $[x_1, \ldots, x_n \such \phi]$ or just $[\phi]$.
The case $n = 0$ yields an element of $\pow{\AA}$.

More precisely, the logical connectives appearing in~$\phi$ are interpreted as the corresponding Heyting operations from \cref{sec:heyt-prealg-struct}.
A universally quantified formula $\all{y \of Y} \psi$, where $\psi$ is a formula in variables $x_1, \ldots, x_n$ and $y$, is interpreted as quantification along the projection
\begin{equation*}
  r : X_1 \times \cdots \times X_n \times Y \to X_1 \times \cdots \times X_n,
\end{equation*}
as in \cref{sec:quantifiers}, and similarly for $\some{y \of Y} \psi$.

\begin{example}
  \label{example:tripos-forall-exists}
  Given a tripos predicate $\psi \in \Pred{X \times Y}$ with a non-empty set~$X$, the formula
  \begin{equation*}
    \all{x \of X} \some{y \of Y} \psi(x,y)
  \end{equation*}
  has no free variables, and so determines an element of $\pow{\AA}$ which, since~$X$ is non-empty,
  may be computed using~\eqref{eq:alternative-forall}:
  \begin{align*}
    \R{a} \in [\all{x \of X} \some{y \of Y} \psi(x,y)]
    &\iff
      \all{u \in X}
      \R{a} \in [\some{y} \psi(u, y)]
    \\
    &\iff
      \all{u \in X}
      \some{v \in Y}
      \R{a} \in \psi(u, v).
  \end{align*}
  Note that $\R{a}$ may not depend on~$u$ and~$v$.
  This is a rather strong uniformity condition, stemming from the fact that realizers receive no information about the elements of underlying sets. When we pass from the tripos to the topos, the situation will be rectified by equipping sets with suitable realizability relations, see \cref{example:topos-forall-exists}.
\end{example}

We say that a formula $\phi$ in variables $x_1 \of X_1, \ldots, x_n \of X_n$ is \defemph{valid}, written as
\begin{equation*}
  x_1 \of X_1, \ldots, x_n \of X_n \validates \phi,
\end{equation*}
when its interpretation is (equivalent to) the top predicate in $\Pred{X_1 \times \cdots \times X_n}$. This happens precisely when there is $\R{a} \in \AA$ such that $\R{a} \rz[p] [\phi](u_1, \ldots, u_n)$ for all $u_1 \in X_1, \ldots, u_n \in X_n$ and $p \in \PP$.

\subsection{The parameterized realizability topos on a ppca}
\label{sec:unif-real-topos}

Having defined a tripos, we employ the tripos-to-topos construction~\cite[Sec.~2.2]{oosten08:_realiz} to construct a topos from it.

\begin{definition}
  The \defemph{parameterized realizability topos} $\PRT{\AA, \PP}$ on the ppca $(\AA, \PP, {\cdot})$ is the topos arising from the tripos-to-topos construction applied to the tripos~$\PredSymbol[\AA, \PP]$.
\end{definition}

We recall how the construction works.
An object $X = (\carrier{X}, {\eq[X]})$ of the topos is a set~$\carrier{X}$ with a tripos predicate ${\eq[X]} \in \Pred{\carrier{X} \times \carrier{X}}$, called the \defemph{equality predicate}, which is a partial equivalence relation in the sense of tripos logic:
\begin{align*}
  x \of \carrier{X}, y \of \carrier{X} &\validates x \eq[X] y \limply y \eq[X] x,
  \\
  x \of \carrier{X}, y \of \carrier{X}, z \of \carrier{X} &\validates x \eq[X] y \limply y \eq[X] z \limply x \eq[X] z.
\end{align*}
Specifically, this means that there are $\R{a}, \R{b} \in \AA$ such that:
\begin{enumerate}
\item for all $x, y \in \carrier{X}$, $\R{c} \in (x \eq[X] y)$, and $p \in \PP$, we have $\R{a} \, \R{c} \rz[p] y \eq[X] x$,
\item for all $x, y, z \in \carrier{X}$, $\R{c} \in (x \eq[X] y)$, $\R{d} \in (y \eq[X] z)$, and $p \in \PP$, we have $\R{b} \, \R{c} \, \R{d} \rz[p] x \eq[X] z$.
\end{enumerate}
Henceforth we shall refrain from such explicit unfolding of formulas into statements about realizers, and instead rely on the fact that a formula is valid in the tripos logic if it has an intuitionistic proof~\cite[Thm.~2.1.6]{oosten08:_realiz}.

The equality predicate $\eq[X]$ endows $\carrier{X}$ with a notion of equality that is witnessed by realizers.
However, because we did not require reflexivity of~$\eq[X]$, there may be elements which are not equal to themselves.
To manage the anomaly, we define the \defemph{existence predicate} $\Ex{X} \in \Pred{\carrier{X}}$ by
\begin{equation*}
  \Ex{X}(x) \defeq (x \eq[X] x).
\end{equation*}
A realizer $\R{a} \in \Ex{X}(x)$ can be thought of as witnessing the fact that $x \in X$. When $\Ex{X}(x) = \emptyset$, the element $x$ ``does not exist'' from the point of view of the topos.
We shall strategically use $\Ex{X}(x)$ to disregard such non-existent elements.\footnote{%
It turns out that~$X$ is isomorphic to $(X', {\eq[X]})$ where $X' \defeq \set{x \in \carrier{X} \such (x \eq[X] x) \neq \emptyset}$, but insisting that $(x \eq[X] x) \neq \emptyset$ does not lead to any improvements.%
}

A morphism $F : X \to Y$ is represented by a predicate $F \in \Pred{\carrier{X} \times \carrier{Y}}$ which is a functional, i.e., one satisfying the following conditions, with $x, x' \of \carrier{X}$ and $y, y' \of \carrier{Y}$:
\begin{align*}
  x, y &\validates F(x,y) \limply \Ex{X}(x) \land \Ex{Y}(y)
     & &\text{(strict)} \\
  x, x', y, y' &\validates F(x,y) \land x \eq[X] x' \land y \eq[Y] y' \limply F(x', y')
     & &\text{(relational)} \\
  x, y, y' &\validates F(x, y) \land F(x, y') \limply y \eq[Y] y'
     & &\text{(single-valued)} \\
  x &\validates \Ex{X}(x) \limply \some{y \of Y} F(x, y)
     & &\text{(total)}
\end{align*}
Single-valuedness and totality are familiar conditions, while the other two ensure that~$F$
is well-behaved with respect to existence and equality predicates. Note how the antecedent $\Ex{X}(x)$ in the totality condition allows~$F$ to ignore non-existing elements of~$X$.
Two such relations represent the same morphism if they are equivalent as tripos predicates.

To actually compute~$F$, we use a realizer~$\R{t}$ for its totality and a realizer~$\R{u}$ for its strictness to define the realizer $\R{c} \defeq \ucode{\abstr{a} \combSnd (\R{u} \, (\R{t} \, a))}$, which works as follows: for any $x \in X$ and $\R{a} \in \Ex{X}(x)$ there is $y \in Y$ such that $\R{c} \, \R{a} \rz[p] \Ex{Y}(y)$ for all~$p \in \PP$,
and because $F$ is single-valued, $y$ is unique up to $\eq[Y]$.

The identity morphism on~$X$ is represented by $\eq[X]$,
and the composition of $F : X \to Y$ and $G : Y \to Z$ by the tripos predicate
\begin{equation*}
  (G \circ F)(x, z) \defeq \some{y \of Y} F(x, y) \land G(y, z).
\end{equation*}
The relevant conditions may be checked by reasoning in intuitionistic logic.

The terminal object in the topos is $\one \defeq (\set{\star}, {\eq[\one]})$, where $(\star \eq[\one] \star) \defeq \AA$.
The subobject classifier is the object $\Omega \defeq (\pow{\AA}, \eq[\Omega])$ whose equality predicate is logical equivalence,
$
  (\phi \eq[\Omega] \psi) \defeq
  (\phi \to \psi) \land (\psi \to \phi)
$.
Truth $T : \one \to \Omega$ is represented by the tripos predicate $T(\star, \phi) \defeq \phi$.

\subsection{Topos logic}
\label{sec:internal-logic-topos}

The topos logic differs from the tripos logic because it accounts for the equality and existence predicates. We refer to~\cite[Sec.~2.3]{oosten08:_realiz} for details, and give here an overview that will suffice for our purposes.

In the topos logic, the predicates on an object $X$ are its subobjects, which turn out to be in
bijective correspondence with equivalence classes of \defemph{strict predicates}~\cite[Thm.~2.2.1]{oosten08:_realiz}, i.e., those $\phi \in \Pred{\carrier{X}}$ that satisfy
\begin{align*}
  x \of X &\validates \phi(x) \limply \Ex{X}(x) & &\text{(strict)} \\
  x \of X, y \of X &\validates \phi(x) \land x \eq[X] y \limply \phi(y) & &\text{(relational)}
\end{align*}
The tripos falsehood is strict, and the tripos conjunction, disjunction, and the existential quantifier preserve strictness, hence these are computed in the topos in the same way as in the tripos. Truth, implication, and the universal quantifier require modification. We distinguish notationally between the tripos and topos logic by writing
``$\limply$'', ``$\forall y \of Y$'', and ``$\exists y \of Y$'' in the former, and
``$\lthen$'',  ``$\forall y \in Y$'', and ``$\exists y \in Y$'' in the latter.

First, the topos truth $\top$ qua predicate on~$X$ is the tripos predicate $\Ex{X}$. Indeed, this is a strict predicate, and for any strict predicate $\phi \in \Pred{X}$ the implication $\phi(x) \to \Ex{X}(x)$ is valid by strictness of~$\phi$. Because the top predicate has changed, we must also adjust validity: a strict predicate $\phi \in \Pred{X}$ is topos-valid when the tripos validates
\begin{equation*}
  x \of X \validates \Ex{X}(x) \to \phi(x).
\end{equation*}
Explicitly, there exists $\R{a} \in \AA$ such that for all $x \in \carrier{X}$, $\R{b} \in \Ex{X}(x)$ and $p \in \PP$ we have $\R{a} \, \R{b} \rz[p] \phi(x)$.

Second, the topos implication $\phi \lthen \psi$ of strict predicates  $\phi$ and $\psi$ on~$X$ is represented by the strict predicate
\begin{equation*}
  [x \of X \mid \Ex{X}(x) \land (\phi(x) \limply \psi(x))].
\end{equation*}
Explicitly, $\R{a} \in (\phi \lthen \psi)(x)$ when for all $p \in \PP$
\begin{equation*}
  (\combFst \, \R{a} \rz[p] \Ex{X}(x))
  \land
  (\combSnd \, \R{a} \rz[p] \phi(x) \to \psi(x)).
\end{equation*}

Third, if $\phi$ is a strict predicate on $X \times Y$, the topos universal $\all{y \in Y} \phi(x, y)$ is represented by the strict predicate
\begin{equation*}
  [x \of X \mid \Ex{X}(x) \land \all{y \of \carrier{Y}} (\Ex{Y}(y) \limply \phi(x,y))].
\end{equation*}
When $\carrier{Y}$ is non-empty we may use \eqref{eq:alternative-forall} to compute that, for any given $x \in \carrier{X}$, we have $\R{a} \in (\all{y \in Y} \phi(x, y))$ when for all $p \in \PP$
\begin{equation*}
  (\combFst \, \R{a} \rz[p] \Ex{X}(x))
  \land
  \all{y \in \carrier{Y}} \all{\R{b} \in \Ex{Y}(y)} \combSnd \, \R{a} \, \R{b} \rz[p] \phi(x, y).
\end{equation*}
The first conjunct just makes sure that non-existent~$x$ do not get in the way. The second one is more interesting, as it adjusts the unreasonable uniformity of tripos~$\forall$ by providing $\combSnd \, \R{a}$ with a realizer of~$y \in \carrier{Y}$.

One might expect the topos existential $\some{y \in Y} \phi(x, y)$ to be
\begin{equation*}
  [x \of X \such \some{y \of Y} \Ex{Y}(y) \land \phi(x, y)],
\end{equation*}
but we can reuse $\some{y \of Y} \phi(x, y)$, for if $\R{a} \in (\some{y \of Y} \phi(x, y))$
and~$s$ realizes strictness of~$\phi$ then $s \, \R{a} \in \Ex{Y}(y)$ for some $y \in \carrier{Y}$.

In contrast to the tripos logic, the topos logic is equipped with equality.
Unsurprisingly, equality on~$X$ is represented by $\eq[X]$, one just needs to check that this is indeed a strict predicate.
More generally, equality of morphisms $F, G : X \to Y$ is represented by the predicate
\begin{equation*}
  [x \of X \such \some{y \of Y} F(x, y) \land G(x, y)].
\end{equation*}

\begin{example}
  \label{example:topos-forall-exists}%
  Suppose $\carrier{X}$ is non-empty, and $\phi \in \Pred{\carrier{X} \times \carrier{Y}}$.
  A short calculations shows that $\all{x \in X} \some{y \in Y} \phi(x, y)$ is realized
  when there is $\R{a} \in \AA$ such that
  \begin{equation*}
    \all{x \in \carrier{X}}
    \all{\R{b} \in \Ex{X}(x)}
    \all{p \in \PP}
    \some{y \in \carrier{Y}}
    \R{a} \, \R{b} \rz[p] \phi(x, y).
  \end{equation*}
  Note that the unreasonable uniformity of \cref{example:tripos-forall-exists} has been rectified,
  as~$\R{b}$ is passed to~$\R{a}$.
\end{example}

\subsection{Parameterized assemblies}
\label{sec:unif-assemblies}

Direct manipulation of topos objects, and especially morphisms, can be cumbersome. Fortunately, the
subcategory of assemblies~\cite[Sec.~2.4]{oosten08:_realiz} is significantly easier to work with and already contains most objects of interest.

The idea is to take existence predicates as primary.
Define a \defemph{(parameterized) assembly} $X = (\carrier{X}, \Ex{X})$ to be a set $\carrier{X}$ with a tripos predicate $\Ex{X} \in \Pred{\carrier{X}}$, called the \defemph{existence predicate}, such that $\Ex{X}(x) \neq \emptyset$ for all $x \in \carrier{X}$.
Also define a \defemph{(parameterized) assembly map} $f : X \to Y$ to be a map $f : \carrier{X} \to \carrier{Y}$ which is realized by some $\R{a} \in \AA$, meaning: for all $x \in \carrier{X}$, $\R{b} \in \Ex{X}(x)$ and $p \in \PP$, we have $\R{a} \, \R{b} \rz[p] \Ex{Y}(f(x))$.
Assembly maps are closed under composition and include the identity maps, so we have a category $\PAsm{\AA, \PP}$. 

Given an assembly~$X$, let $\eq[X]$ be the tripos predicate on $\carrier{X} \times \carrier{X}$, defined by
\begin{equation*}
  (x \eq[X] x') \defeq \set{\R{a} \in \Ex{X}(x) \such x = x'}.
\end{equation*}
Thus $x \eq[X] x'$ is empty when $x \neq x'$ and equals $\Ex{X}(x)$ when $x = x'$.
It is evident that $x \eq[X] x'$ is an equality predicate on~$\carrier{X}$, hence the assembly~$X$ may be construed as the topos object $(\carrier{X}, \eq[X])$.
Not every topos object arises this way, for instance the subobject classifier~$\Omega$ does not.

To get a functorial embedding of assemblies into the topos, we map an assembly map $f : X \to Y$ to the topos morphism $F : X \to Y$ where
\begin{equation*}
  F(x, y) \defeq \set{\R{b} \in \AA \such
    f(x) = y
    \land
    \all{p \in \PP}
    \combFst \, \R{b} \rz[p] \Ex{X}(x)
    \land
    \combSnd \, \R{b} \rz[p] \Ex{Y}(y)}.
\end{equation*}
This is a functional relation, for if~$\R{a}$ realizes~$f$ then $\ucode{\abstr{x} \combPair \, x \, (\R{a} \, x)}$ realizes totality of~$F$.
The passage from assemblies to topos objects constitutes a full and faithful embedding $\PAsm{\AA, \PP} \to \PRT{\AA, \PP}$. Only fullness deserves attention. Suppose $X$ and $Y$ are assemblies and $F : X \to Y$ a morphism between the induced topos objects. Because $Y$ is an assembly and $F$ is single-valued, each $x \in \carrier{X}$ has at most one $y \in \carrier{Y}$ such that $F(x, y) \neq \emptyset$---but it also has at least one because~$F$ is total. Therefore, we may define a map $f : \carrier{X} \to \carrier{Y}$ by
\begin{equation*}
  f(x) = y \defiff F(x,y) \neq \emptyset.
\end{equation*}
If $\R{a} \in \AA$ realizes totality of $F$ then $\ucode{\abstr{x} \combSnd \, (\R{a} \, x)}$ realizes~$f$ as an assembly map.

\subsection{Some distinguished assemblies}
\label{sec:distinguished-assemblies}

We review certain objects of the topos that will play a role in the construction of the object of the Dedekind reals.

\subsubsection{Natural numbers, integers, and rational numbers}
\label{sec:natur-numb-integ}

The natural numbers object is the assembly $\objN \defeq (\NN, \Ex{\objN})$ where $\Ex{\objN}(n) \defeq \set{\numeral{n}}$, so that each number is realized by the corresponding Curry numeral.
The induction principle is realized by the primitive recursor $\comb{primrec}$ from \cref{sec:progr-with-ppcas}.

The objects of integers and rational numbers are the assemblies
\begin{equation*}
  \objZ \defeq (\ZZ, \Ex{\objZ})
  \quad\text{and}\quad
  \objQ \defeq (\QQ, \Ex{\objQ}),
\end{equation*}
whose existence predicates are induced by computable enumerations.
For the integers we can use
\begin{equation*}
  \Ex{\objZ}(k) \defeq
  \begin{cases}
    \set{\numeral{2 k}}     & \text{if $k \geq 0$,} \\
    \set{\numeral{1 - 2 k}} & \text{if $k < 0$.}
  \end{cases}
\end{equation*}
For the rationals we may reuse the bijection $\rat{} : \NN \to \QQ$ from \cref{sec:oracle-comp-maps},
\begin{equation*}
  \Ex{\objQ}(\rat{n}) \defeq \set{\numeral{n}}.
\end{equation*}
We also define the assembly $\two \defeq (\{0,1\},\Ex{\two})$, where $\Ex{\two}(0) = \{\numeral{0}\}$ and $\Ex{\two}(1) = \{\numeral{1}\}$.

Any other reasonable codings would result in isomorphic objects.
Arithmetical operations are realized and the order relation is decidable, i.e., the statement
${\all{x, y} x < y \lor y \leq x}$ is realized, both for $x, y \in \ZZ$ and for $x, y \in \QQ$.

\subsubsection{Products and exponentials}
\label{sec:prod-expon}

The category of parameterized assemblies is cartesian closed. The product of $X$ and $Y$ is the assembly
\begin{equation*}
  X \times Y \defeq (\carrier{X} \times \carrier{Y}, \Ex{X \times Y})
\end{equation*}
where
\begin{equation*}
  \Ex{X \times Y}(x, y) \defeq
  \set{\R{a} \in \AA \such
  \all{p \in \PP} \combFst \, \R{a} \rz[p] \Ex{X}(x) \land \combSnd \, \R{a} \rz[p] \Ex{Y}(y)}.
\end{equation*}
\begin{sloppypar}
  To construct the exponential~$Y^X$, we define its existence predicate, for any ${f : \carrier{X} \to \carrier{Y}}$, by
\end{sloppypar}
\begin{equation*}
  \Ex{Y^X}(f) \defeq
  \set{\R{a} \in \AA \such
    \all{x \in X} \all{\R{b} \in \Ex{X}(x)} \all{p \in \PP}
      \R{a} \, \R{b} \rz[p] \Ex{Y}(f(x))
  },
\end{equation*}
and set $\carrier{Y^X} \defeq \set{f : \carrier{X} \to \carrier{Y} \such \Ex{Y^X}(f) \neq \emptyset}$.

\begin{proposition}
  \label{prop:markov-principle}%
  Markov's principle
  \begin{equation*}
    \all{f \in \two^\objN} \neg \neg (\some{n \in \objN} f n = 1) \lthen \some{n \in \objN} f n = 1
  \end{equation*}
  is valid.
\end{proposition}

\begin{proof}
  The principle is realized by a program that searches for the least $n$ such that $f n \neq 0$:
  \begin{equation*}
    \abstr{f} \abstr{r}
    \comb{Z} \, (\abstr{s} \abstr{n}
         \combIf \,
         (\comb{iszero} \,
         (f \, n) \,
         (s \, (\comb{succ} \, n)) \,
         n))
     \, \numeral{0}.
  \end{equation*}
  The assumption $\neg\neg{\some{n \in \objN} f(n) = 1}$ ensures that the search will succeed.\footnote{As is typical of realizability models, we are relying on meta-level Markov's principle to realize Markov's principle: since it is impossible that the search will run forever, it will find what it is looking for.}
\end{proof}

\begin{example}
  Let us contrast $\forall\exists$ statements and exponentials. Consider a non-empty assembly~$X$, an assembly~$Y$, and
  a strict predicate $\phi \in \Pred{\carrier{X} \times \carrier{Y}}$, with $s \in \AA$ witnessing its strictness.
  Validity of $\all{x \in X} \some{y \in Y} \phi(x,y)$ is equivalent to there being $\R{a} \in \AA$ such that
  \begin{equation*}
    \all{x \in \carrier{X}} \all{\R{b} \in \Ex{X}(x)} \all{p \in \PP} \some{y \in \carrier{Y}} \R{a} \, \R{b} \rz[p] \phi(x,y).
  \end{equation*}
  The realizer $\R{c} \defeq \ucode{\abstr{\R{b}} \combSnd \, (s \, (\R{a} \, \R{b}))}$ satisfies, for any $x \in \carrier{X}$, $\R{b} \in \Ex{X}(x)$ and $p \in \PP$, that there is $y \in \carrier{Y}$ such that $\R{c} \, \R{b} \rz[p] \Ex{Y}(y)$. However, $\R{c}$ need not realize a choice map $f : X \to Y$ because~$y$ may depend on~$\R{b}$ and~$p$.
  Thus in general the axiom of choice
  \begin{equation*}
    (\all{x \in X} \some{y \in Y} \phi(x, y))
    \lthen
    \some{f \in Y^X} \all{x \in X} \phi(x, f(x)))
  \end{equation*}
  is not realized, even in case that $\Ex{X}(x)$ is a singleton for all~$x \in \carrier{X}$, because there is still dependence on the parameter. In particular, countable choice may fail, as it does in the topos~$\TT{\mil}$ from \cref{sec:topos-with-countable}.
\end{example}

\subsubsection{Sub-assemblies}
\label{sec:sub-assemblies}

Suppose $\phi \in \Pred{\carrier{X}}$ is a strict predicate on an assembly~$X$. (Notice that $\phi$ is automatically relational because $(x \eq[X] y) \neq \emptyset$ implies $x = y$.) We define the sub-assembly $\set{x \of X \such \phi(x)}$ to have the underlying set
\begin{equation*}
  \carrier{\set{x \of X \such \phi(x)}} \defeq \set{x \in \carrier{X} \such \phi(x) \neq \emptyset}
\end{equation*}
and the existence predicate $\Ex{\set{x \of X \such \phi(x)}}(x) \defeq \phi(x)$.
Then the canonical map $\set{x \of X \such \phi(x)} \to X$ is realized by any realizer for strictness of~$\phi$.
It is the monomorphism classified by the predicate~$\phi$.

\subsubsection{Constant assemblies}
\label{sec:constant-assemblies}

Define the \defemph{constant (parameterized) assembly} on a set $S$ to be $\nabla S \defeq (S, \Ex{\nabla S})$ with $\Ex{\nabla S}(x) \defeq \AA$.
The existence predicate is maximally uninformative, because all elements of~$S$ share all realizers.
Consequently, given any assembly $X$, every map $f : \carrier{X} \to S$ is realized, say by~$\combK$.
In particular, every map $f : S \to T$ between sets is an assembly map $\nabla f \defeq f : \nabla S \to \nabla T$,
which makes $\nabla$ a functor from sets to assemblies, see~\cite[Sec.~2.4]{oosten08:_realiz} for details.

\subsubsection{The $\neg\neg$-stable predicates and the assembly $\nabla\two$}
\label{sec:assembly-nabl-negn}

A predicate $\phi \in \Pred{X}$ on an object~$X$ is said to be \defemph{$\neg\neg$-stable} when $x \of X \validates \neg\neg\phi(x) \lthen \phi(x)$.
The assembly $\nabla \two$ classifies such predicates.
On the one hand, $(\nabla\two)^X$ is isomorphic to $\nabla(\two^{\carrier{X}})$ because every map $\carrier{X} \to \two$ is realized as a morphism $X \to \nabla\two$.
On the other, $\two^{\carrier{X}}$ qua Heyting algebra is equivalent to the Heyting prealgebra of $\neg\neg$-stable strict predicates on~$X$. To see this, observe that a strict predicate~$\phi$ on~$X$ is $\neg\neg$-stable when
\begin{equation*}
  x \of X \validates \Ex{X}(x) \to ((\phi(x) \lthen \bot) \lthen \bot) \lthen \phi(x),
\end{equation*}
%
%
%
%
%
which amounts to there being $\R{a} \in \AA$ such that, for all $x \in \carrier{X}$, $\R{b} \in \Ex{X}(x)$ and $p \in \PP$,
if $\phi(x) \neq \emptyset$ then $\R{a} \, \R{b} \rz[p] \phi(x)$.
Therefore, $\phi$ is equivalent to the strict predicate
$x \mapsto \set{a \in \AA \such \phi(x) \neq \emptyset}$,
which in turn corresponds to a unique map $\carrier{X} \to \two$, obtained when~$\emptyset$ and~$\AA$ are mapped to $0$ and~$1$, respectively.


\section{The real numbers}
\label{sec:real-numbers-object}

We review the construction of the Dedekind real numbers, and formulate it in a way that makes it easy to calculate the object of Dedekind reals in a parameterized realizability topos.
We also show that the Cauchy reals are sequence-avoiding in any parameterized realizability topos.

\subsection{The Dedekind real numbers}
\label{sec:dedek-real-numb}

In this section we work in higher-order intuitionistic logic~\cite{jim86:_introd_higher_order_categ_logic}, which can be interpreted in any topos.
Common mathematical constructions are available, as well as the standard number sets: the natural numbers $\objN$, the integers~$\objZ$ and the rationals~$\objQ$.

\begin{definition}
  \label{def:dedekind-reals}%
  A \defemph{Dedekind cut} is a pair $(L, U) \in \pow{\objQ} \times \pow{\objQ}$ of subsets of rationals, satisfying the following conditions, where $q$ and~$r$ range over~$\objQ$:
  \begin{enumerate}
  \item $L$ and $U$ are inhabited: $\some{q} q \in L$ and $\some{r} r \in U$,
  \item $L$ is lower-rounded and $U$ upper-rounded:
    \begin{equation*}
      q \in L \liff \some{r} q < r \land r \in L 
      \qquad\text{and}\qquad
      r \in U \liff \some{q} q \in U \land q < r,
    \end{equation*}
  \item $L$ is below $U$: $q \in L \land r \in U \lthen q < r$,
  \item $L$ and $U$ are located: $q < r \lthen q \in L \lor r \in U$.
  \end{enumerate}
  We write $\Cut{L,U}$ for the conjunction of the above conditions.
  The set of \defemph{Dedekind reals} is
  \begin{equation*}
    \RRd \defeq \set{ (L, U) \in \pow{\objQ} \times \pow{\objQ} \such \Cut{L, U} }.
  \end{equation*}
\end{definition}

The symbols $\RRd$ and $\RR$ both denote the set of Dedekind reals. We normally use $\RRd$ when referring to the object of Dedekind reals in a topos, or when we want to contrast the Dedekind reals with other kinds of reals.

The bi-implication defining lower-roundedness can be split into two separate conditions:
\begin{itemize}
\item $L$ is lower: $q < r \land r \in L \lthen q \in L$,
\item $L$ is rounded: $q \in L \lthen \some{r} q < r \land r \in L$.
\end{itemize}
Upper-roundedness may be decomposed analogously.

Many textbooks construct the reals by using one-sided cuts. Again, this works classically but requires special care when done constructively, and in any case the symmetry of double-sided cuts streamlines the development, even classically. 

\begin{propositionC}
  \label{prop:RRd-stable-equality}%
  For all $x, y \in \RRd$, if $\lnot\lnot (x = y)$ then $x = y$.
\end{propositionC}

\begin{proof}
  Suppose $x = (L_x, U_x)$, $y = (L_y, U_y)$ and $\lnot\lnot(x = y)$.
  We only prove $L_x \subseteq L_y$, as the other three inclusions are proved symmetrically.
  Suppose $q \in L_x$. Because $L_x$ is rounded, there is $r \in \QQ$ such that $q < r \in L_x$.
  From $\lnot\lnot (L_x = L_y)$ follows that $\lnot\lnot (r \in L_y)$.
  Because~$y$ is located, $q \in L_y$ or $r \in U_y$. In the first case we are done, while the second case cannot happen, for if $r \in U_y$ then $\lnot (r \in L_y)$, contradicting $\lnot\lnot (r \in L_y)$.
\end{proof}

\begin{propositionC}
  \label{prop:stradle-closelyd}
  For any cut $(L, U)$ and $k \in \NN$ there exists $q \in \QQ$ such that $q - 2^{-k} \in L$ and $q + 2^{-k} \in U$.
\end{propositionC}

\begin{proof}
  There are $s \in L$ and $t \in U$. Let us show by induction on $j \in \NN$ that there are $u \in L$ and $v \in U$ such that $v - u \leq (2/3)^j (t - s)$. At $j = 0$ we take $u = s$ and $v = t$.
  The induction step from $j$ to $j+1$ proceeds as follows. By the induction hypothesis there are $u' \in L$ and $v' \in U$ such that $v' - u' \leq (2/3)^j (t - s)$. By locatedness $(2 u' + v')/3 \in L$ or $(u' + 2 v')/3 \in U$. In the first case we take $u = (2 u' + v')/3$ and $v = v'$, and in the second $u = u'$ and $v = (u' + 2 v')/3$.

  Now let $k \in \NN$ be given. There is $j \in \NN$ such that $(2/3)^j (t - s) < 2^{-k}$. We proved that there exist $u \in L$ and $v \in U$ such that $v - u \leq (2/3)^j (t - s) < 2^{-k}$. We may take $q = (u + v)/2$ because
  $q - 2^{-k} < v - 2^{-k} < u$ hence $q - 2^{-k} \in L$, and similarly for $q + 2^{-k} \in U$.
\end{proof}

To facilitate calculations in parameterized realizability, we find an object that is isomorphic to~$\RRd$ but whose interpretation in assemblies is straightforward.
For any set $A$ and subset $S \subseteq A$, let $\compl{S} \defeq \set{x \in A \such \neg (x \in S)}$ be the complement of~$S$, and
\begin{equation*}
  \powcl{A} \defeq \set{S \in \pow{A} \such \all{x \in A} \neg\neg(x \in S) \lthen x \in S}
\end{equation*}
the set of $\neg\neg$-stable subsets of~$A$. In the same way that $\pow{A}$ is isomorphic to the set $\Omega^A$ of characteristic maps on~$A$, $\powcl{A}$ is isomorphic to $\ClProp^A$, where
\begin{equation*}
  \ClProp \defeq \set{p \in \Omega \such \neg\neg p \lthen p}
\end{equation*}
is the set of $\neg\neg$-stable truth values. $\ClProp$ is a complete Boolean algebra that can be used instead of~$\Omega$ in the definition of Dedekind cuts, as follows:

\begin{definition}
  A \defemph{classical Dedekind cut} is a pair $(L, U) \in \powcl{\objQ} \times \powcl{\objQ}$
  of $\neg\neg$-stable subsets of rationals, satisfying the following conditions, where $r$ and $q$ range over~$\objQ$:
  \begin{enumerate}
  \item $L$ and $U$ are not empty: $\lnot \all{q} q \not\in L$ and $\lnot \all{r} r \not\in U$,
  \item $L$ is lower and $U$ upper:
    \begin{equation*}
      q < r \land r \in L \lthen q \in L
      \quad\text{and}\quad
      q \in U \land q < r \lthen r \in U,
    \end{equation*}
  \item $L$ has no maximum and $U$ no minimum:
    \begin{equation*}
      (\all{r} r \in L \lthen r \leq q) \lthen q \not\in L
      \quad\text{and}\quad
      (\all{q} q \in U \lthen r \leq q) \lthen r \not\in U,
    \end{equation*}
  \item $L$ is below $U$: $q \in L \land r \in U \lthen q < r$,
  \item $L$ and $U$ are tight: $q \not\in L \land r \not\in U \lthen r \leq q$.
  \end{enumerate}
  We write $\ClCut{L, U}$ for the conjunction of the above conditions.
  The set of \defemph{classical Dedekind reals} is
  \begin{equation*}
    \RRcl \defeq \set{ (L, U) \in \powcl{\objQ} \times \powcl{\objQ} \such \ClCut{L, U}}.
  \end{equation*}
\end{definition}

Dedekind cuts turn out to be those classical Dedekind cuts that have arbitrarily good rational approximations.

\begin{theoremC}
  \label{thm:reals-sub-classical}
  The set of Dedekind reals $\RRd$ is isomorphic to
  \begin{equation}
    \label{eq:reals-sub-classical}%
    R_d \defeq \set{ (L, U) \in \RRcl \such 
       \all{k \in \NN} \some{q \in \QQ} q - 2^{-k} \in L \land q + 2^{-k} \in U
     }.
   \end{equation}
\end{theoremC}

\begin{proof}
  Let $f : \RRd \to R_d$ take a cut to its double complement, $f(L, U) = (\dcompl{L}, \dcompl{U})$.
  To see that it is well-defined, we must verify that $(\dcompl{L}, \dcompl{U})$ is a classical cut and that
  \begin{equation}
    \label{eq:mvv-reals}
    \all{k \in \NN} \some{q \in \QQ} q - 2^{-k} \in \dcompl{L} \land q + 2^{-k} \in \dcompl{U}.
  \end{equation}
  We dispense with~\eqref{eq:mvv-reals} first: for any $k \in \NN$, by \cref{prop:stradle-closelyd} there is $q \in \QQ$ such that $q - 2^{-k} \in L \subseteq \dcompl{L}$ and $q + 2^{-k} \in U \subseteq \dcompl{U}$.

  Let us verify that $(\dcompl{L}, \dcompl{U})$ is a classical cut, where we spell out only the conditions for $L$, as the reasoning for $U$ is symmetric:
  \begin{enumerate}
  \item $\dcompl{L}$ is non-empty because $L$ is inhabited and $L \subseteq \dcompl{L}$.
  \item $\dcompl{L}$ is lower: if $q < r$ and $r \in \dcompl{L}$ then $q \in \compl{L}$ would imply $r \in \compl{L}$, hence $q \in \dcompl{L}$.
  \item $\dcompl{L}$ has no maximum: if  $\all{r \in \dcompl{L}} r \leq q$ then $\all{r \in L} r \leq q$, hence $q \in \compl{L} = \compl{(\dcompl{L})}$.
  \item $\dcompl{L}$ is below $\dcompl{U}$: if $q \in \dcompl{L}$ and $r \in \dcompl{U}$ then $\lnot\lnot(q < r)$, therefore $q < r$.
  \item $\dcompl{L}$ and $\dcompl{U}$ are tight: suppose $q \not\in \dcompl{L}$ and $r \not\in \dcompl{U}$. If $q < r$ then either $q \in L$, which contradicts $q \not\in \dcompl{L}$, or $r \in U$, which contradicts $q \not\in \dcompl{U}$. Therefore $r \leq q$.
  \end{enumerate}

  It remains to be shown that $f$ is a bijection. For injectivity, suppose $x = (L_x, U_x)$ and $y = (L_y, U_y)$ are cuts such that $\dcompl{L_x} = \dcompl{L_y}$ and $\dcompl{U_x} = \dcompl{U_y}$. Then $\lnot\lnot (x = y)$ and by 
  \cref{prop:RRd-stable-equality}, $x = y$.

  To establish surjectivity of~$f$, take any $(L, U) \in R_d$ and define
  \begin{align*}
    \hat{L} &\defeq \set{q \in \QQ \such \some{r \in \QQ} q < r \land r \in L},
    \\
    \hat{U} &\defeq \set{r \in \QQ \such \some{q \in \QQ} q \in U \land q < r}.
  \end{align*}
  Let us show that $(\hat{L}, \hat{U})$ is a cut, where again we verify only the conditions for~$\hat{L}$ when symmetry permits us to do so:
  \begin{enumerate}
  \item $\hat{L}$ is inhabited, because by~\eqref{eq:mvv-reals} there is $q \in \QQ$ such that $q - 2^{0} \in L$, therefore $q - 2 \in \hat{L}$.
  \item $\hat{L}$ is obviously lower, and it is rounded too: if $q \in \hat{L}$ then $q < r$ and $r \in L$ for some $r$,
    hence $(q + r)/2 \in \hat{L}$ because $(q + r)/2 < r$.
  \item $\hat{L}$ is below $\hat{U}$: if $q \in \hat{L}$ and $r \in \hat{U}$ then there are $s$ and $t$ such that $q < s$, $s \in L$, $t \in U$, and $t < r$, therefore $q < s < t < r$.
  \item To see that $(\hat{L},  \hat{U})$ is located, consider any $q < r$. There is $k \in \NN$ such that $2^{-k+1} < r - q$, and there is $s \in \QQ$ such that $s - 2^{-k} \in L$ and $s + 2^{-k} \in U$. Either $q < s - 2^{-k}$, in which case $q \in \hat{L}$, or $s + 2^{-k} < r$, in which case $r \in \hat{U}$.
  \end{enumerate}
  Finally, we claim that $\dcompl{\hat{L}} = L$ and $\dcompl{\hat{U}} = U$.
  Notice that $\hat{L} \subseteq L$ because $L$ is lower, hence $\dcompl{\hat{L}} \subseteq \dcompl{L} = L$.
  To show $L \subseteq \dcompl{\hat{L}}$, suppose $q \in L$. If we had $q \in \compl{\hat{L}}$, then it would follow that $\all{r \in L} r \leq q$, and because $L$ has no maximum also $q \not\in L$, a contradiction, hence we conclude $q \in \dcompl{\hat{L}}$.
\end{proof}

\subsection{The Dedekind reals in parameterized realizability}
\label{sec:dedek-reals-param}

We seek an explicit description of the object of Dedekind reals in a parameterized realizability topos. Rather than interpreting \cref{def:dedekind-reals} directly, we compute the classical Dedekind cuts and use~\cref{thm:reals-sub-classical}.

\begin{proposition}
  \label{prop:PRT-classical-reals}%
  In a parameterized realizability topos, the object of classical Dedekind reals is isomorphic to $\nabla\RR$.
\end{proposition}

\begin{proof}
  In \cref{sec:assembly-nabl-negn} we saw that $\nabla\two$ is isomorphic to~$\ClProp$,
  hence $\powcl{\objQ}$ is isomorphic to ${\nabla\two}^\objQ$, which in turn is isomorphic to $\nabla\left(\pow{\QQ}\right)$.
  Because the definition of $\ClCut{L,U}$ uses only logical connectives and relations that
  are preserved by~$\nabla$, the object $\RRcl$ is isomorphic to
  \begin{equation*}
    \nabla \set{(L, U) \in \pow{\QQ} \times \pow{\QQ} \such \ClCut{L, U}}.
  \end{equation*}
  We are done because the classical definition of~$\RR$ appears under~$\nabla$.
\end{proof}

\begin{corollary}
  \label{cor:dedekind-characterization}%
  In $\PRT{\AA, \PP}$ the object of Dedekind reals is (isomorphic to) the assembly $\RRd$ whose underlying set is
  $
    \carrier{\RRd} = \set{x \in \RR \such \Ex{\RRd}(x) \neq \emptyset}
  $
  and whose existence predicate is
  \begin{equation*}
    \Ex{\RRd}(x) \defeq
    \set{\R{r} \in \AA \such
      \all{p \in \PP}
      \all{k \in \NN}
      \some{q \in \QQ}
      \abs{x - q} < 2^{-k} \land
      \R{r} \, \numeral{k} \rz[p] \Ex{\objQ}(q)
    }.
  \end{equation*}
\end{corollary}

\begin{proof}
  By \cref{prop:PRT-classical-reals,thm:reals-sub-classical} the object of Dedekind reals is (isomorphic to) the
  sub-assembly of $\nabla{\RR}$ whose existence predicate is the realizability interpretation of~\eqref{eq:reals-sub-classical}, which is what~$\Ex{\RRd}$ is.
\end{proof}

The takeaway is that a realizer of a Dedekind real~$x$ computes arbitrarily precise rational approximations of~$x$ which \emph{may} depend on the parameter~$p$.

Recall that the strict order $<$ on $\RRd$ is defined by
\begin{equation*}
  x < y \defiff
  \some{q \in \QQ} q \in U_x \land q \in L_y.
\end{equation*}
Thus $x < y$ is realized by a rational that is wedged between~$x$ and~$y$.
However, the following lemma shows that we need not bother because the order is $\neg\neg$-stable.

\begin{lemma}
  \label{lem:lt-stable}%
  In $\PRT{\AA, \PP}$,
  \begin{enumerate}
  \item $\all{x, y \in \RRd} x \neq y \lthen x < y \lor y < x$, and
  \item $\all{x, y \in \RRd} \neg \neg (x < y) \lthen x < y$.
  \end{enumerate}
\end{lemma}

\begin{proof}
  Deriving the second statement from the first one is an exercise in intuitionistic logic,
  so we only describe how to realize the first statement, which is sometimes referred to as the
  \defemph{analytic Markov's principle}~\cite{shulman18:brouw}.
  Suppose $x, y \in \carrier{\RRd}$, $\R{r} \in \Ex{\RRd}(x)$, $\R{t} \in \Ex{\RRd}(y)$ and~$p \in \PP$.
  Search for the least~$k \in \NN$ such that $\abs{\rat{\R{r} \app[p] k} - \rat{\R{t} \app[p] k}} > 2^{-k+1}$.
  It will certainly be found because $x \neq y$.
  With~$k$ in hand, compare $\rat{\R{r} \app[p] k}$ and $\rat{\R{t} \app[p] k}$.
  If $\rat{\R{r} \app[p] k} < \rat{\R{t} \app[p] k}$ then
  \begin{equation*}
    x < \rat{\R{r} \app[p] k} + 2^{-k}
      < \rat{\R{t} \app[p] k} - 2^{-k}
      < y,
  \end{equation*}
  therefore $x < y$ may be realized by $\rat{\R{r} \app[p] k} + 2^{-k}$.
  If $\rat{\R{t} \app[p] k} < \rat{\R{r} \app[p] k}$ then $y < x$ symmetrically.
\end{proof}

\subsection{The Cauchy reals}
\label{sec:cauchy-reals}
It is instructive to compare the Dedekind and Cauchy reals in parameterized realizability.
We identify the Cauchy reals as those Dedekind reals that are limits of Cauchy sequences of rationals. Since we work without the axiom of countable choice, rapid convergence should be imposed.

\begin{definition}
  \label{def:cauchy-real}
  A \defemph{Cauchy real} is a Dedekind real $x \in \RRd$ which is the limit of a rapidly converging rational sequence, i.e., the set of Cauchy reals is
  \begin{equation*}
    \RRc \defeq \set{x \in \RRd \such \some{q \in \objQ^\objN} \all{n \in \objN} \abs{x - q_n} < 2^{-n}}.
  \end{equation*}
\end{definition}

It would make no difference if we used the classical reals:

\begin{propositionC}
  \label{prop:cauch-real-sub-classical}%
  The set of Cauchy reals $\RRc$ is isomorphic to
  \begin{equation}
    \label{eq:cauchy-characterization}%
    R_c \defeq \set{x \in \RRcl \such \some{q \in \objQ^\objN} \all{n \in \objN} \abs{x - q_n} < 2^{-n}}.
  \end{equation}
\end{propositionC}


\begin{proof}
  Any classical real $x \in \RRcl$ that is the limit of a rapidly converging rational sequence is also a Dedekind real.
\end{proof}

The previous proposition tells us how to compute the object of Cauchy reals.

\begin{corollary}
  \label{cor:cauchy-characterization}%
  In $\PRT{\AA, \PP}$ the object of Cauchy reals is (isomorphic to) the assembly $\RRc$ whose
  underlying set is
  $
    \carrier{\RRc} \defeq \set{x \in \RR \such \Ex{\RRc}(x) \neq \emptyset}
  $
  and the existence predicate is
  \begin{multline*}
    \Ex{\RRc}(x) \defeq
    \{\R{r} \in \AA \such
      \some{q \in \QQ^\NN}
      \all{p \in \PP}
      \all{k \in \NN}
      \abs{x - q_k} < 2^{-k} \land
      \R{r} \, \numeral{k} \rz[p] \Ex{\objQ}(q_k)
   \}.
  \end{multline*}
\end{corollary}

\begin{proof}
  As in \cref{cor:dedekind-characterization}, the existence predicate $\Ex{\RRc}$ is the realizability interpretation of \eqref{eq:cauchy-characterization}.
\end{proof}

If we write the existence predicate for the Dedekind reals in the equivalent form\footnote{The equivalence with the definition given in \cref{cor:dedekind-characterization} relies on countable choice \emph{outside} the topos.}
\begin{multline*}
  \Ex{\RRd}(x) \defeq
  \{\R{r} \in \AA \such
    \all{p \in \PP}
    \some{q \in \QQ^\NN}
    \all{k \in \NN}
    \abs{x - q_k} < 2^{-k} \land
    \R{r} \, \numeral{k} \rz[p] \Ex{\objQ}(q_k)
  \},
\end{multline*}
the difference between the Dedekind and Cauchy reals is seen to be just one of switching the quantifiers:
a rapid sequence representing a Dedekind real may depend on the parameter, whereas one representing a Cauchy real may not.
\begin{theorem}
  \label{thm:cauchy-uncountable}%
  In $\PRT{\AA,\PP}$ the Cauchy reals are sequence-avoiding.
\end{theorem}

\begin{proof}
  For $\R{r} \in \AA$ to realize
  \begin{equation*}
    \all{f \in \RRc^\objN}
    \some{x \in \RRc}
    \all{n \in \objN}
    f(n) \neq x
  \end{equation*}
  it has to satisfy the following condition:
  for all $f \in \carrier{\RRc^\objN}$, $\R{b} \in \Ex{\RRc^\objN}(f)$ and $p \in \PP$
  there is $x \in \carrier{\RRc}$ such that $\all{n \in \NN} f(n) \neq x$ and
  \begin{equation*}
    \some{t \in \QQ^\NN}
    \all{k \in \NN}
    \abs{x - t_k} < 2^{-k}
    \land
    \all{p' \in \PP} (\R{r} \app[p] \R{b}) \app[p'] \numeral{k} \in \Ex{\objQ}(t_k).
  \end{equation*}
  So suppose $f \in \carrier{\RRc^\objN}$, $\R{b} \in \Ex{\RRc^\objN}(f)$ and $p \in \PP$.
  By unraveling the meaning of $\R{b} \in \Ex{\RRc^\objN}(f)$ we find out that
  there is a map $s : \NN \times \NN \to \QQ$, which depends only on~$\R{b}$ and~$p$,
  such that
  \begin{equation*}
    \all{n, k \in \NN}
    \abs{f(n) - s(n, k)} < 2^{-k}
    \land
    \all{p' \in \PP}
    (\R{b} \app[p] \numeral{n}) \app[p'] \numeral{k} \in \Ex{\objQ}(s(n, k)).
  \end{equation*}
  \begin{sloppypar}
    We define sequences $t, u : \NN \to \QQ$ which depend only on~$s$, such that ${0 < u_n - t_n < 2^{-n-1}}$ and $f(n) < u_n \lor t_n < f(n)$, for all $n \in \NN$. Set $t_0 \defeq s(0,0) + 1$ and $u_0 \defeq t_0 + \frac{1}{4}$.
    Assuming $t_n$ and $u_n$ have been constructed, let $d \defeq \frac{u_n - t_n}{4}$, $m \defeq \min \set{m \in \NN \such 2^{-m} < d}$, and define
  \end{sloppypar}
  \begin{equation*}
    (t_{n+1}, u_{n+1}) \defeq
    \begin{cases}
      (t_n, t_n + d) &\text{if $t_n + 2 d < s(n, m)$,} \\
      (u_n - d, u_n) &\text{otherwise.}
    \end{cases}
  \end{equation*}
  %
  %
  The real $x \defeq \lim_n t_n = \lim_n u_n$ satisfies $\all{n \in \NN} f(n) \neq x$ and
  depends only on $\R{b}$ and $p$, because it is defined in terms of~$s$.
  It remains to exhibit~$\R{r} \in \AA$, independent of all parameters, such that $(\R{r} \app[p] \R{b}) \app[p'] \numeral{k} \in \Ex{\objQ}(t_k)$ for all $k \in \NN$ and $p' \in \PP$. It takes the form $\R{r} \defeq \abstr{b} \abstr{k} e$ where~$e$ computes a realizer for~$t_k$, as described in the above procedure, and relying on the fact that $(\R{b} \app[p] \numeral{n}) \app[p'] \numeral{k}$ computes a realizer for~$s(n, k)$.
\end{proof}

Let us investigate where the proof of \cref{thm:cauchy-uncountable} gets stuck if we replace the Cauchy reals with the Dedekind reals.
For $\R{r} \in \AA$ to realize
\begin{equation*}
  \all{f \in \RRd^\objN}
  \some{x \in \RRd}
  \all{n \in \objN}
  f n \neq x
\end{equation*}
it has to satisfy the following condition:
for all $f \in \carrier{\RRd^\objN}$, $\R{b} \in \Ex{\RRd^\objN}(f)$ and $p \in \PP$
there is $x \in \carrier{\RRd}$ such that $\all{n \in \NN} f(n) \neq x$ and
\begin{equation*}
  \all{k \in \NN}
  \all{p' \in \PP}
  \some{t \in \QQ}
  \abs{x - t} < 2^{-k}
  \land
  (\R{r} \app[p] \R{b}) \app[p'] k \in \Ex{\objQ}(t).
\end{equation*}
So suppose $f \in \carrier{\RRd^\objN}$, $\R{b} \in \Ex{\RRd^\objN}(f)$ and $p \in \PP$.
By unraveling the meaning of $\R{b} \in \Ex{\RRd^\objN}(f)$ we find out that
for every $p' \in \PP$ there is a map $s : \NN \times \NN \to \QQ$, which depends on~$\R{b}$ and~$p$ \emph{as well as~$p'$}, such that
\begin{equation*}
  \all{n, k \in \NN}
  \abs{f(n) - s(n, k)} < 2^{-k}
  \land
  (\R{b} \app[p] n) \app[p'] k \in \Ex{\objQ}(s(n, k)).
\end{equation*}
At this point we are stuck: if we continue by constructing a sequence $t : \NN \to \QQ$ as in the proof, its limit $\lim_n t_n$ will depend on~$p'$, but we need a real~$x$ that does not.
For a specific ppca $(\AA, \PP)$ one might find some way of constructing~$x$ that avoids dependence on~$p'$,
although there can be such no general construction as that would contradict~\cref{thm:countable-reals}.

Finally, let us show that in $\PRT{\AA, \PP}$ the carrier sets of Cauchy and Dedekind reals coincide.
Recall that $x \in \RRd$ is \defemph{strongly irrational}\footnote{In intuitionistic mathematics this is a stronger notion than being \defemph{irrational}, which means $x \neq q$ for all~$q \in \objQ$.} when $x < q \lor q < x$ for all $q \in \objQ$.

\begin{lemmaC}
  \label{lem:rat-irrat-cauchy}
  Every rational and every strongly irrational Dedekind real is a Cauchy real.
\end{lemmaC}

\begin{proof}
  Clearly, every rational number is a Cauchy real.
  Suppose $x \in \RRd$ is strongly irrational. Then we may find a rational sequence converging rapidly to~$x$ by simple bisection, as follows. There are rational $t_0$ and $u_0$ such that $t_0 < x < u_0$ and $u_0 - t_0 < 1$.
  Given rational $t_n < x < u_n$ such that $u_n - t_n < 2^{-n}$, let $m \defeq \frac{t_n + u_n}{2}$ and define
  \begin{equation*}
    (t_{n+1}, u_{n+1}) =
    \begin{cases}
      (t_n, m) & \text{if $x < m$,} \\
      (m, u_n) & \text{if $m < x$.}
    \end{cases}
  \end{equation*}
  Then both $t$ and $u$ are rapid rational sequences converging to~$x$.
\end{proof}

\begin{proposition}
  \label{prop:carrier-Rd-eq-Rc}%
  In a parameterized realizability topos, $\carrier{\RRd} = \carrier{\RRc}$.
\end{proposition}

\begin{proof}
  It suffices to show internally that there is no Dedekind real which is not a Cauchy real.
  Suppose, for contradiction, that~$x$ is such a real. Then by \cref{lem:rat-irrat-cauchy}, $x$ is neither rational nor strongly irrational.
  However, \cref{lem:lt-stable} implies that every irrational real is strongly irrational.
  Thus~$x$ would be a real that is neither rational nor irrational, a contradiction.
\end{proof}

To summarize, the difference between $\RRc$ and $\RRd$ is not one of extent but one of \emph{structure}:
the identity map $x \mapsto x$ is realized as a morphism $\RRc \to \RRd$, but not necessarily as a morphism $\RRd \to \RRc$.


\section{A topos with countable reals}
\label{sec:topos-with-countable}

Given a Miller sequence~$\mil$, as in~\cref{sec:non-diag-sequ}, let~$\MM{\mil} \defeq \invim{\srep}(\mil)$ be the set of all oracles representing~$\mil$ where $\srep$ is the representing map from \cref{sec:oracle-comp-maps}. Define $\TT{\mil} \defeq \PRT{\KK,\MM{\mil}}$ to be the parameterized realizability topos constructed from the ppca $\KK$ with oracles $\MM{\mil}$, see \cref{ex:oracle-ppca}.

\subsection{Countability of the reals}
\label{sec:countability-reals}
Let us immediately address countability of the Dedekind reals in~$\TT{\mil}$. 
We first reduce the problem to countability of the closed interval.

\begin{lemmaC}
  \label{lem:R-contable-iff-I-countable}%
  The real numbers are countable if, and only if, the closed unit interval is countable.
\end{lemmaC}

\begin{proof}
  If $\RR$ is countable then $[0,1]$ is countable because there is a retraction $\RR \to [0,1]$, for instance
  $x \mapsto \max(0, \min(1, x))$.
  Conversely, given a surjection $e : \NN \to [0,1]$, we claim that $e' : \NN \to \RR$, defined by
  $e'(\pair{m, n}) \defeq m \cdot (2 \cdot e(n) - 1)$, is a surjection also. For any $x \in \RR$ there is $m \in \NN$ such that
  $-m < x < m$, and there is $n \in \NN$ such that $e(n) = \frac{x + m}{2 m}$, hence $e'(\pair{m, n}) = x$.
  %
\end{proof}

Next, we obtain custom descriptions of the assemblies of natural numbers and the closed unit interval.

\begin{lemma}
  \label{lem:nno-assembly}%
  In the topos $\TT{\mil}$, the natural numbers object is isomorphic to the assembly~$\objN$ with carrier
  $\carrier{\objN} \defeq \NN$ and the existence predicate
  $\Ex{\objN}(n) \defeq \set{n}$.
\end{lemma}

\begin{proof}
  In \cref{sec:natur-numb-integ} we saw that the natural numbers object is the assembly with carrier
  $\NN$ and existence predicate $n \mapsto \set{\numeral{n}}$. The assembly from the statement is isomorphic
  to it because we may convert between $\numeral{n}$ and $n$ using the combinators~$\combNum$ and~$\combCur$ from \cref{ex:numers-vs-numerals}.
\end{proof}

Henceforth we use~$\objN$ from \cref{lem:nno-assembly} as the standard natural numbers object. The practical consequence is that we may eschew Curry numerals and instead use numbers directly.

\begin{lemma}
  \label{lem:interval-assembly}%
  In the topos $\TT{\mil}$, the closed unit interval is isomorphic to the assembly~$\objI$ with carrier
  $\carrier{\objI} \defeq [0,1] \cap \carrier{\RRd}$ and the existence predicate
  $\Ex{\objI}(x) \defeq \set{m \in \KK \such \mil(m) = x}$.
\end{lemma}

\begin{proof}
  The sub-assembly $\set{x \in \RRd \such 0 \leq x \land x \leq 1}$ has $[0,1] \cap \carrier{\RRd}$ as its carrier set. Its existence predicate is the tripos predicate
  $[x \of \RRd \such 0 \leq x \land x \leq 1]$,
  which is $\neg\neg$-stable, and therefore equivalent to $\Ex{\RRd}(x)$ restricted to~$\carrier{\objI}$.
  Thus, it suffices to show that the tripos logic validates
  \begin{equation}
    \label{eq:objI-ER-to-EI}
    x \of \carrier{\objI} \validates \Ex{\RRd}(x) \to \Ex{\objI}(x).
  \end{equation}
  and
  \begin{equation}
    \label{eq:objI-EI-to-ER}
    x \of \carrier{\objI} \validates \Ex{\objI}(x) \to \Ex{\RRd}(x).
  \end{equation}
  By \cref{cor:dedekind-characterization}, $\R{r} \in \Ex{\RRd}(x)$ is equivalent to
  \begin{equation}
    \label{eq:objI-rz-x}
    \all{\alpha \in \MM{\mil}}
    \all{k \in \NN}
    \abs{x - \rat{\pr[\alpha]{\R{r}}(k)}} < 2^{-k},
  \end{equation}
  where we used $\objN$ from \cref{lem:nno-assembly}.
  Condition \eqref{eq:objI-rz-x} states that~$\R{r}$ is a $\mil$-index for~$x$ in the sense of \cref{def:sequence-computable}, therefore $\mil(\R{r}) = x$ and we may realize \eqref{eq:objI-ER-to-EI} with $\ucode{\abstr{r} r}$.

  It remains to realize~\eqref{eq:objI-EI-to-ER}, which is accomplished by the realizer $\wrep \in \NN$, as characterized by~\eqref{eq:wrep}.
  If~$\alpha$ codes $\mil$ then $\pr[\alpha]{\wrep}(m)$ represents $\mil(m)$ for all $m \in \NN$,
  therefore~$\wrep$ realizes~\eqref{eq:objI-EI-to-ER}.
\end{proof}

\begin{theorem}
  \label{thm:countable-reals}
  In the topos $\TT{\mil}$, there is an epimorphism from natural numbers to Dedekind reals.
\end{theorem}

\begin{proof}
  By \cref{lem:R-contable-iff-I-countable,lem:interval-assembly} it suffices to show that the Miller sequence $\mil : \objN \to \objI$, which is realized by $\ucode{\abstr{m} m}$, is an epimorphism. This is so because
  \begin{equation*}
    \all{x \in \objI} \some{m \in \objN} \mil(m) = x
  \end{equation*}
  is trivially realized by $\ucode{\abstr{m}{m}}$ as well.
  %
  %
\end{proof}

\subsection{What else is countable?}
\label{sec:what-else-countable}
Given that \cref{thm:fixed-point-R-uncountable,thm:countable-reals,thm:cauchy-uncountable} sandwich the countable Dedekind reals between uncountable Cauchy reals and uncountable MacNeille reals, it is natural to wonder which classically uncountable spaces are countable in~$\TT{\mil}$.

Products, sums and images of countable sets are countable, which gives basic examples of countable spaces, such as Euclidean spaces $\RRd^n$, hypercubes~$\objI^n$, the unit circle $T \defeq \set{(x,y) \in \RRd \times \RRd \such x^2 + y^2 = 1}$, $n$-spheres, etc.

William F.~Lawvere's fixed-point theorem~\cite{lawvere69} is a source of uncountable sets.

\begin{theoremC}[Lawvere]
  \label{thm:lawvere}
  If $e : A \to B^A$ is surjective then every $f : B \to B$ has a fixed point.
\end{theoremC}

\begin{proof}
  Because $e$ is surjective,
  there is $a \in A$ such that $e(a) = (x \mapsto f(e(x)(x)))$, whence $e(a)(a) = f(e(a)(a))$.
\end{proof}

As soon as there is a fixed-point free map $X \to X$, there is no surjection ${\objN \to X^\objN}$, by the contrapositive of Lawvere's theorem. We already noted in \cref{cor:cantor-diagonal} this to be the case for Cantor space.
Two further examples are the countable powers $\RRd^\objN$ and $T^\objN$ of the Dedekind reals and the unit circle, which are uncountable because (non-trivial) translations of~$\RRd$ and rotations of~$T$ have no fixed points.

How about the Hilbert cube $\objI^\objN$?
One might attempt to enumerate it by composing $\mil : \objN \to \objI$ with a space-filling curve $\objI \to \objI^\objN$. However, even constructing just a square-filling curve $\objI \to \objI \times \objI$ in~$\TT{\mil}$ seems impossible,\footnote{One cannot intuitionistically construct a square-filling curve $[0,1] \to [0,1] \times [0,1]$ because there is no such curve in the topos of sheaves on the closed unit square. We do not know whether there is a square-filling curve in~$\TT{\mil}$.}
 so we take another route.
We first need a lemma showing that the elements of $\objI^\objN$ take a special form.

\begin{lemma}
  \label{lem:flattening-realizers}
  There is a total computable function $\ell : \NN \times \NN \to \NN$ such that $\mil(\ell(m,n)) = f(n)$ for all $f \in \carrier{\objI^\objN}$, $m \in \Ex{\objI^\objN}(f)$, and $n \in \NN$.
\end{lemma}

\begin{proof}
  In \cref{sec:oracle-comp-maps} we obtained $\wrep \in \NN$, as characterized by \eqref{eq:wrep}, such that $\pr[\alpha]{\wrep}(n)$ is the code of a map representing~$\mil(n)$ for all~$\alpha \in \MM{\mil}$ and~$n \in \NN$. The map $\ell : \NN \times \NN \to \NN$,
  \begin{equation*}
    \ell(m, n) \defeq \ucode{\abstr{x} \wrep \, (m \, n) \, x},
  \end{equation*}
  is well-defined by \cref{lem:abstr-uniform} and is computable.

  Now consider any $f \in \carrier{\objI^\objN}$ and $m \in \Ex{\objI^\objN}(f)$.
  Given any $n \in \NN$, we establish $\mil(\ell(m, n)) = f(n)$ by verifying that $\rcomp{\mil}{\ell(m,n)} = f(n)$.
  For any $\alpha \in \MM{\mil}$ and $k \in \NN$,
  \begin{align*}
    \pr[\alpha]{\ell(m,n)}(k)
    &\kleq \alpha \at (\abstr{x} \wrep \, (m \, n) \, x) \, k \\
    &\kleq \alpha \at \wrep \, (m \, n) \, k \\
    &\kleq
         \pr[\alpha]{
           \pr[\alpha]{\wrep}(
             \pr[\alpha]{m}(n)
           )
         }(k),
  \end{align*}
  therefore $\pr[\alpha]{\ell(m,n)} = \pr[\alpha]{\pr[\alpha]{\wrep}(\pr[\alpha]{m}(n))}$, which is a representation
  of $f(n)$.
\end{proof}

A curious consequence of \cref{lem:flattening-realizers} is that every $f \in \carrier{\objI^\objN}$ is equal to $\mil\circ g$ for some total $g : \NN \to \NN$ that is computable without an oracle.

\begin{theorem}
  \label{thm:hilbert-countable}%
  In the topos $\TT{\mil}$, the Hilbert cube~$\objI^\objN$ is countable.
\end{theorem}

\begin{proof}
  Let $\ell : \NN \times \NN \to \NN$ be as in Lemma~\ref{lem:flattening-realizers} and
  $\R{l} \in \NN$ a realizer for~$\ell$, which exists because~$\ell$ is computable.
  The map $e : \carrier{\objN} \to \carrier{\objI^\objN}$, defined by $e(m)(n) \defeq \mil(\ell(m,n))$, is realized by
  $\ucode{\abstr{m} \abstr{n} \R{l} \, (\combPair \, m \, n)}$.
  To show that $e$ is an epimorphism it suffices to prove that
  \begin{equation*}
    \all{f \in \objI^\objN}
    \some{m \in \objN}
    \all{n \in \objN}
    \mil(\ell(m,n)) = f(n)
  \end{equation*}
  is realized by~$\ucode{\abstr{x} x}$.
  By unfolding the realizability interpretation we find that this amounts to
  \begin{equation*}
    \all{\alpha \in \MM{\mil}}
    \all{\R{b} \in \Ex{\objI^\objN}(f)}
    \all{n \in \NN}
    \R{b} \, n \rz[\alpha] \mil(\ell(\R{b},n)) = f(n).
  \end{equation*}
  This is indeed true by \cref{lem:flattening-realizers} and the fact that~$\R{b}$ realizes~$f$.
\end{proof}

\section{\texorpdfstring{Mathematics in the topos~$\TT{\mil}$}{Mathematics in the topos Tμ}}
\label{sec:analysis-topos-tt}

We devote the last section to exploring a little further the peculiar new mathematical world~$\TT{\mil}$.

\subsection{Brouwer's fixed-point theorem}
\label{sec:brouwers-fixed-point}
The reader may have noticed already that having a surjection $\objN \to \objI^\objN$ is precisely the antecedent of Lawvere's theorem, which we can use to give a short proof of Brouwer's fixed-point theorem.

\begin{theorem}[Brouwer's fixed-point theorem]
  \label{thm:internal-brouwer}%
  In the topos $\TT{\mil}$, every map $\objI^\objN \to \objI^\objN$ has a fixed point,
  and so does every map $\objI^n \to \objI^n$, for every $n \in \NN$.
\end{theorem}

\begin{proof}
  Combining \cref{thm:hilbert-countable} and Lawvere's \cref{thm:lawvere} yields a fixed point of any map $f : \objI \to \objI$. When $n = 0$ the statement is trivial. For the remaining cases, note that the evident bijections $\objN \to \objN \times \objN$ and $\objN \to \set{1, \ldots, n} \times \objN$ induce bijections $\objI^\objN \to (\objI^\objN)^\objN$ and $\objI^\objN \to (\objI^n)^\objN$.
  Composing these with the surjection $e : \objN \to \objI^\objN$ yields surjections $\objN \to (\objI^\objN)^\objN$ and $\objN \to (\objI^n)^\objN$, respectively, so Lawvere's theorem applies again.
\end{proof}

In~\cref{sec:comp-clos-interv} we shall use the following variant of Brouwer's fixed point theorem for partial maps with $\neg\neg$-stable domains of definition.

\begin{theorem}
  \label{thm:partial-brouwer}%
  For every $n \in \NN$, the topos $\TT{\mil}$ validates
  \begin{equation*}
    \all{\phi \in \objI^n \to \ClProp}
    \all{f \in {\set{x \in \objI^n \such \phi(x)}} \to \objI^n}
    \some{y \in \objI^n}
    \phi(y) \lthen f(y) = y.
  \end{equation*}
\end{theorem}

\begin{proof}
  We demonstrate the proof for $n = 2$, which is the instance used in \cref{prop:drinking-buddies}.
  We seek $\R{r} \in \NN$ such that, for all $\alpha \in \MM{\mil}$,
  $\phi : \carrier{\objI^2} \to \set{\bot, \top}$, and
  $f : \set{x \in \carrier{\objI^2} \such \phi(x)} \to \carrier{\objI^2}$
  with $\R{f} \in \NN$ satisfying
  \begin{equation*}
    \all{x \in \carrier{\objI^2}}
    \all{\R{x} \in \Ex{\objI^2}(x)}
    \all{\beta \in \MM{\mil}}
    \phi(x) \lthen (\beta \at \R{f} \, \R{x}) \in \Ex{\objI^2}(f(x)),
  \end{equation*}
  there is $y \in \carrier{\objI^2}$ such that $(\alpha \at \R{r} \, \R{f}) \in \Ex{\objI^2}(y)$ and if $\phi(y)$ then $f(y) = y$.

  The most direct way of computing a potential fixed point of~$f$ would employ~$\R{r}$ such that
  $\R{r} \, \R{f} \simeq \R{f} \, (\R{r} \, \R{f})$. This is too simplistic, but making sure that~$\R{f}$
  is applied to a pair of realizers representing reals gives us enough control.
  Specifically, we shall use the fact that $\ucode{\abstr{k} \wrep \, \R{m} \, k} \in \Ex{\objI}(\mil(\R{m}))$
  for all $\R{m} \in \NN$, where~$\wrep$ is the realizer from~\eqref{eq:wrep}.
  Recalling the fixed-point combinator~$\comb{Z}$ we thus define
  \begin{align*}
    \R{r} &\defeq \ucode{\comb{Z} \,
                (\abstr{s} \abstr{g} g \,
                    (\combPair \,
                      (\abstr{k} \wrep \, (\combFst \, (s \, g)) \, k) \,
                      (\abstr{k} \wrep \, (\combSnd \, (s \, g)) \, k)%
                    )
                )},\\
    \R{a} &\defeq \ucode{\abstr{k} \wrep \, (\combFst \, (\R{r} \, \R{f})) \, k},\\
    \R{b} &\defeq \ucode{\abstr{k} \wrep \, (\combSnd \, (\R{r} \, \R{f})) \, k},
  \end{align*}
  %
  %
  %
  For any $\alpha \in \MM{\mil}$, the fixed-point equation for~$\comb{Z}$ yields
  $(\alpha \at \R{r} \, \R{f}) \kleq (\alpha \at \R{f} \, (\combPair \, \R{a} \, \R{b}))$.

  We claim that $y \defeq (\mil(\R{a}), \mil(\R{b})) \in \carrier{\objI}^2$, which is realized by $\combPair\,\R{a}\,\R{b}$, satisfies $\phi(y) \lthen f(y) = y$.
  To establish the claim, assume $\phi(y)$, so that
  $f(y) = (u, v)$ for unique $u, v \in \carrier{\objI}$.
  Consider any $\alpha \in \MM{\mil}$.
  Because $(\alpha \at \R{f}\,(\combPair\,\R{a} \, \R{b})) \in \Ex{\objI^2}((u, v))$ and
  $\alpha \at \R{f} \, (\combPair \, \R{a} \, \R{b}) = \alpha \at \R{r} \, \R{f}$, it follows that
  $(\alpha \at \combFst \, (\R{r} \, \R{f})) \in \Ex{\objI}(u)$ and
  $(\alpha \at \combSnd \, (\R{r} \, \R{f})) \in \Ex{\objI}(v)$.
  Moreover, $\alpha \at \combFst \, (\R{r} \, \R{f})$ realizes the same real as~$\R{a}$ and
  $\alpha \at \combSnd \, (\R{r} \, \R{f})$ the same real as~$\R{b}$, therefore $\mil(\R{a}) = u$ and $\mil(\R{b}) = v$, from which $y = f(y)$ follows.
\end{proof}

\subsection{The intermediate value theorem}
\label{sec:interm-value-theor}
The 1-dimensional Brouwer's fixed-point theorem and the intermediate value theorem are derivable from each other.

\begin{lemmaC}
  \label{lem:max-neq-eq}%
  For any $a, b \in \RR$, if $\max(a, b) \neq a$ then $\max(a, b) = b$, and similarly for $\min$.
\end{lemmaC}

\begin{proof}
  If $\max(a, b) \neq a$ then $\neg (b \leq a)$.
  By \cref{prop:RRd-stable-equality} it suffices to show that $\neg (\max(a, b) \neq b)$.
  If $\max(a, b) \neq b$ then $\neg (a \leq b)$, which together with $\neg (b \leq a)$ yields a contradiction.
\end{proof}

\begin{theorem}[Intermediate value theorem]
  \label{thm:ivt}%
  In the topos $\TT{\mil}$, if $f : \objI \to \RRd$ satisfies $f(0) < 0 < f(1)$ then $f(x) = 0$ for some $x \in \objI$.
\end{theorem}

\begin{proof}
  Given such an~$f$, define $g : \objI \to \objI$ by
  $g(x) \defeq \max(0, \min(1, x - f(x)))$.
  By \cref{thm:internal-brouwer} there is $x \in \objI$ such that
  \begin{equation*}
    \max(0, \min(1, x - f(x))) = x.
  \end{equation*}
  Use \cref{lem:max-neq-eq} to derive $\min(1, x - f(x)) = x - f(x)$ and once more to derive $x - f(x) = x$,
  yielding $f(x) = 0$.
\end{proof}

\subsection{Real maps do not jump}
\label{sec:continuity-maps}
Say that $f : \RR \to \RR$ \defemph{jumps at $x$} if there is $\epsilon > 0$ such that $\abs{f(y) - f(x)} > \epsilon$ for all $y > x$. Countability of reals is at odds with existence of such explicitly discontinuous maps.

\begin{propositionC}
  \label{prop:real-maps-jump-real-dec}%
  If there is a map with a jump then $x = y$ or $x \neq y$ for all $x, y \in \RR$.
\end{propositionC}

\begin{proof}
  Without loss of generality we consider $f : \RR \to \RR$ such that~$f(0) = 0$ and $f(z) > 1$ for all $z > 0$.
  Given any $x, y \in \RR$, either $f(\abs{x - y}) < \sfrac{2}{3}$ or $f(\abs{x-y}) > \sfrac{1}{3}$.
  In the former case we obtain $\neg (\abs{x - y} > 0)$, hence $x = y$ by \cref{prop:RRd-stable-equality}.
  In the latter case we obtain $\neg (\abs{x - y} = 0)$, hence $x \neq y$.
\end{proof}

\begin{corollaryC}
  \label{prop:real-maps-jump}%
  If there is a map with a jump then $\RR$ is uncountable.
\end{corollaryC}

\begin{proof}
  By \cref{prop:real-maps-jump-real-dec} equality on the reals would be decidable, therefore Cantor's method
  of nested intervals from the proof of \cref{thm:R-uncountable} could be employed. Specifically, we could decide
  whether the first option in~\eqref{eq:R-uncountable} holds.
\end{proof}

Thus in the topos~$\TT{\mil}$, equality of reals is not decidable and there are no maps with jumps.
An alternative reason for the same fact is this: a map with a jump may be used to violate the intermediate value theorem.

Does~$\TT{\mil}$ validate the stronger statement ``all maps $\RRd \to \RRd$ are continuous''?
In recursive analysis this was proved by Kreisel, Lacombe, Shoenfield~\cite{KreiselLacombeShoenfield59} and Tseitin~\cite{Tseitin67}. Our attempts to adapt their proof to~$\TT{\mil}$ failed, so we must leave the question unanswered.

\subsection{A modicum of classical logic}
\label{sec:modic-class-logic}

Brouwer's fixed-point theorem is a fairly unusual property for an intuitionistic topos to have because it implies a constructive taboo, namely the so-called \emph{analytic Limited Lesser Principle of Omniscience (LLPO)}.

\begin{theorem}[Analytic LLPO]
  \label{thm:llpo}%
  In the topos $\TT{\mil}$, every Dedekind real is non-negative or non-positive.
\end{theorem}

\begin{proof}
  Given any $x \in \RRd$, the map $y \mapsto \max(0, \min(1, y + x))$ has a fixed point~$y \in \objI$.
  Either $\sfrac{1}{3} < y$ or $y < \sfrac{2}{3}$.
  Taking into account that $\neg (x < 0) \lthen x \geq 0$ and $\neg (x > 0) \lthen x \leq 0$ are intuitionistically valid, we resolve both cases easily.
  If $\sfrac{1}{3} < y$ then $x \geq 0$, because $x < 0$ would imply $y = \max(0, \min(1, y + x)) = \max(0, y + x) < y$.
  If $y < \sfrac{2}{3}$ then $x \leq 0$, because $x > 0$ would imply $y = \max(0, \min(1, y + x)) = \min(1, y + x) > y$.
\end{proof}

The non-analytic LLPO states that if $a : \objN \to \two$ attains value~$1$ at most once, then either $\all{n} a_{2 n} = 0$ or $\all{n} a_{2 n + 1} = 0$. This follows from \cref{thm:llpo}: if $\sum_{n} a_n \cdot (- \sfrac{1}{2})^n$ is non-positive then $\all{n} a_{2 n} = 0$ and if it is non-negative then $\all{n} a_{2 n + 1} = 0$.
This is as far as coquetry with classical logic goes, for the stronger \defemph{Limited Principle of Omniscience (LPO)}
\begin{equation*}
  \all{a \in \two^\NN} (\some{n \in \NN} a_n = 1) \lor (\all{n \in \NN} a_n = 0)
\end{equation*}
fails, as we show next.

\begin{lemma}[Majority Decision Principle]
  \label{lem:majority-decision}%
  In the topos $\TT{\mil}$, the following holds:
  for all propositions $p_1, p_2, p_3, q \in \Omega$, if
  \begin{enumerate}
  \item $\neg p_i \lthen \neg q \lor \neg \neg q$ for $i = 1, 2, 3$, and
  \item $p_i \lthen \neg p_j$ for $i \neq j$,
  \end{enumerate}
  then $\neg q \lor \neg \neg q$.
\end{lemma}

We first explain the principle informally, from a computational perspective.
Suppose we are given three decision algorithms for deciding the same statement, but any one might be corrupted so that
it gives the wrong answer, or no answer at all. If we interpret $p_i$ as ``the $i$-th algorithm is corrupted'', then the first condition states that a non-corrupted algorithm works correctly, and the second that if one of the algorithms is corrupted then the other two are not. Under these circumstances, we can make the decision by running all three algorithms in parallel and waiting for two of them to agree.
  With this intuition in mind, we describe a program that realizes the principle.

\begin{proof}
  Suppose realizers $\R{r}_i$ of $\neg p_i \lthen \neg q \lor \neg \neg q$ are given for $i = 1, 2, 3$,
  as well as $\alpha \in \MM{\mil}$.
  If a negated proposition is realized, then it is realized by~$\combK$.
  Therefore, the first condition guarantees that if $\neg p_i$ is realized, then
  $\R{r}_i \, \combK \rz[\alpha] \neg q \lor \neg\neg q$, in which case
  \begin{itemize}
  \item if $\alpha \at \combFst \, (\R{r}_i \, \combK) = \ucode{\combTrue}$ then
    $\combPair \, \combTrue \, \combK \rz[\alpha] \neg q \lor \neg\neg q$, and
  \item if $\alpha \at \combFst \, (\R{r}_i \, \combK) = \ucode{\combFalse}$ then
    $\combPair \, \combFalse \, \combK \rz[\alpha] \neg q \lor \neg\neg q$.
  \end{itemize}
  We interleave the computations $\alpha \at \combFst \, (\R{r}_i \, \combK)$ for $i = 1, 2, 3$, and as soon as two of them agree on an answer~$\R{b}$, we output $\combPair \, \R{b} \, \combK$.
  By the second condition, an agreement will be reached and will be correct.
\end{proof}

One might hope to generalize the conclusion of the Majority Decision Principle to~$\neg q \lor q$, but it is unclear how, for if two computations seem to realize the right disjunct, we would not know which one yielded a valid realizer for~$q$.

\begin{lemmaC}
  \label{lem:lpo-irrat-deceq}
  If LPO holds, then $x = y$ or $x \neq y$ for all strongly irrational~$x$ and~$y$.
\end{lemmaC}

\begin{proof}
  For strongly irrational $x$ and $y$ and a rational~$q$, we have
  \begin{equation*}
    (q < x \liff q < y) \lor \neg (q < x \liff q < y).
  \end{equation*}
  Now LPO implies
  \begin{equation*}
    (\all{q \in \QQ} q < x \liff q < y)
    \lor
    \neg (\all{q \in \QQ} q < x \liff q < y),
  \end{equation*}
  which is equivalent to $x = y \lor x \neq y$.
\end{proof}

\begin{theoremC}
  \label{thm:majority-lpo-deceq}%
  If the Majority Decision Principle, the Analytic Markov's Principle,\footnote{The principle was formulated in \cref{lem:lt-stable}.} and LPO hold then equality of reals is decidable.
\end{theoremC}

\begin{proof}
  The proof idea is due to Joseph Miller.
  It suffices to prove $x \neq 1 \lor x = 1$ for all positive $x \in \RR$, so consider a positive~$x \in \RR$.
  If $x \sqrt{2}$ is irrational, then it is strongly irrational by the Analytic Markov's Principle,
  hence \cref{lem:lpo-irrat-deceq} yields $x \sqrt{2} = \sqrt{2} \lor x \sqrt{2} \neq \sqrt{2}$,
  and $x = 1 \lor x \neq 1$ follows.
  The same holds if we replace~$2$ with $3$ or~$5$.
  Next, if $x \sqrt{2}$ is rational then $x \sqrt{3}$ and $x \sqrt{5}$ are irrational, and similarly with the roles of the numbers rotated. Thus, we may apply the Majority Decision Principle, \cref{lem:majority-decision}, with~$q$ set to $x = 1$ and $p_1, p_2, p_3$ set respectively to $x \sqrt{2} \in \QQ$, $x \sqrt{3} \in \QQ$, $x \sqrt{5} \in \QQ$.
  This gives $x \neq 1 \lor \neg(x \neq 1)$, which is equivalent to $x \neq 1 \lor x = 1$ by \cref{prop:RRd-stable-equality}.
\end{proof}

\begin{corollary}
  In the topos $\TT{\mil}$, the Limited Principle of Omniscience fails.
\end{corollary}

\begin{proof}
  The topos validates both the Majority Decision Principle and the Analytic Markov's Principle. If it also validated LPO, then \cref{thm:majority-lpo-deceq} would imply that the reals have decidable equality, but we showed in \cref{sec:continuity-maps} that they do not.
\end{proof}

\subsection{Compactness of the closed interval}
\label{sec:comp-clos-interv}
In constructive mathematics various classically equivalent notions of compactness diverge~\cite{bridges02:_compac_contin_const_revis}.
We focus on the Heine-Borel compactness of the closed unit interval, which states that every open cover has a finite subcover, as it is the most interesting one in the topos~$\TT{\mil}$.

Say that a sequence of open intervals $(a_0, b_0), (a_1, b_1), \ldots$ forms a \defemph{singular cover} of $[0,1]$ if it covers the interval, but the sum of lengths $\sum_{i \in \NN} b_i - a_1$ is less than~$1$. 
Of course, such a thing does not exist classically. In the topos~$\TT{\mil}$ it is readily manufactured from the enumeration $\mil : \objN \to \objI$, just take any $0 < \epsilon < 1$ and set
\begin{equation*}
  (a_i, b_i) \defeq (\mil_i - \epsilon \cdot 2^{-i-1}, \mil_i + \epsilon \cdot 2^{-i-1}).
\end{equation*}
The $i$-th interval covers $\mil_i$ and $\sum_{i \in \NN} b_i - a_1 = \epsilon$.
Consequently, the Heine-Borel property fails strongly.

\begin{theoremC}
  \label{thm:singular-cover}%
  If $(a_0, b_0), (a_1, b_1), \ldots$ is a singular cover of~$[0,1]$ then for every $n \in \NN$ the set $[0,1] \setminus \bigcup_{i < n} (a_i, b_i)$ is inhabited.
\end{theoremC}

\begin{proof}
  We prove the following stronger statement:
  if $[a_0, b_0], \ldots, [a_{n-1}, b_{n-1}]$ are closed intervals and $(u, v)$ is an open interval such that $\sum_{i=0}^{n-1} b_i - a_i < v - u$, then there is $x \in (u, v)$ which is not in any $[a_i, b_i]$.

  Suppose first that all the endpoints are rational numbers, so that comparisons between them are decidable.
  For each $k = 0, \ldots, n$ we compute a list of pairwise disjoint intervals with rational endpoints
  \begin{equation*}
    (u_{k,1}, v_{k,1}), \ldots, (u_{k, m_k}, v_{k, m_k})
  \end{equation*}
  such that $m_k > 0$ and $\sum_{i=k}^{n-1} b_i - a_i < \sum_{j=1}^{m_k} v_{k,j} - u_{k,j}$, and
  \begin{equation*}
    \textstyle
    (u, v) \setminus \bigcup_{i < k} [a_i, b_i] = \bigcup_{j=1}^{m_k} (u_{k,j}, v_{k,j}).
  \end{equation*}
  Start with $m_0 \defeq 1$ and $(u_{0,1}, v_{0,1}) = (u, v)$.
  To progress to $(k+1)$-th stage, replace each $(u_{k,j}, v_{k,j})$ with the difference $(u_{k,j}, v_{k,j}) \setminus [a_{k}, b_{k}]$, which is a disjoint union of zero, one, or two open intervals.
  Note that the total length decreases by at most $b_{k} - a_{k}$.
  In the end we may take the midpoint of $(u_{n,1}, v_{n,1})$ to be the desired~$x$.

  When the endpoints are real numbers we may slightly enlarge each $[a_i, b_i]$ to an interval with rational endpoints\footnote{Doing so requires only finitely many choices, which can be carried out by induction, without appealing to the axiom of choice.} and slightly shrink $(u,v)$ to an interval with rational endpoints, while preserving
  $\sum_{i=0}^{n-1} b_i - a_i < v - u$.
\end{proof}

The situation is reminiscent of the effective topos~\cite[Thm.~4.2 in Ch.~6]{troelstra88:_const_mathem}, except that there one has to work harder to construct a singular cover because the closed unit interval is not countable. Also note that the sum of lengths of the intervals constructed in \cref{thm:singular-cover} is exactly~$\epsilon$, whereas in the effective topos the sum fails to converge, but its partial sums are bounded by~$\epsilon$.

This however is not all that can be said about the Heine-Borel compactness of the closed unit interval in~$\TT{\mil}$.
Observe that the singular cover constructed above consists of intervals whose endpoints are real numbers.
Can we also construct one whose endpoints are rational? Surprisingly, no.
To see why this is the case we need a bit of preparation.

To lay the groundwork for the proof of \cref{lem:fix-point-free-map}, we explain how to constructively extend certain maps to larger domains. Let
\begin{equation*}
  \psquare = \set{(x,y) \in \usquare \such \max(\abs{2 x - 1}, \abs{2 y - 1}) = 1}
\end{equation*}
be the boundary of the unit square,\footnote{Constructively the given boundary need not be equal to the union of the four sides of the square.} and suppose given continuous maps $f : [a,b] \to \psquare$ and $g : [b, c] \to \psquare$ such that $f(b) = g(b)$. Then they have a common continuous extension $h : [a,c] \to \psquare$, namely
\begin{equation}
  \label{eq:glue-abutting}
  h(x) \defeq f(\min(x, b)) + g(\max(x, b)) - f(b),
\end{equation}
where addition and subtraction are carried out coordinate-wise.
Indeed, $h$ maps into $\psquare$ because it is continuous and does so on the dense subset $[a,b] \cup [b,c] \subseteq [a,c]$.
The construction can be iterated to give an extension of any finite number of matching maps defined on abutting closed intervals.

Next, consider a solid rectangle~$ABCD$ and continuous maps $f : AB \to \psquare$, $g : BC \to \psquare$ and $h : CD \to \psquare$ defined on three of its sides, satisfying $f(B) = g(B)$ and $g(C) = h(C)$. We wish to construct a common continuous extension $j : ABCD \to \psquare$ of~$f$, $g$, and~$h$.
Draw points $A'$, $D'$, $E$ so that $AB \cong A'B$ and $CD \cong CD'$, as in \cref{fig:rectangle}.
\begin{figure}[hbt]
  \centering
  \begin{tikzpicture}[scale=0.85]
    \fill[color=white!90!black] rectangle (2,3) ;
    \draw (2,0) -- (2,3) ;
    \draw[thick]
       (2,3) node[anchor=west] {$D$} --
       (0,3) node[anchor=east] {$C$} --
       (0,0) node[anchor=east] {$B$} --
       (2,0) node[anchor=west] {$A$} ;
    \draw (0, -2) node[anchor=east] {$A'$} -- (0,0) ; \fill (0, -2) circle (0.05) ;
    \draw (0, 5) node[anchor=east] {$D'$} -- (0,3) ; \fill (0, 5) circle (0.05) ;
    \fill (3.5, 1.5) node[anchor=north west] {$E$} circle (0.05) ;
    \fill (1.25, 1.95) circle (0.05) node[anchor=north west] {$P$} ;
    \fill (0, 2.2) circle (0.05) node[anchor=east] {$r(P)$} ;
    \draw (3.5, 1.5) -- (0, 2.2) ;
    \draw[thin,dashed] (3.5, 1.5) -- (0, 5) ;
    \draw[thin,dashed] (3.5, 1.5) -- (0, -2) ;
  \end{tikzpicture}
  \caption{Extending maps from three sides to the rectangle}
  \label{fig:rectangle}
\end{figure}
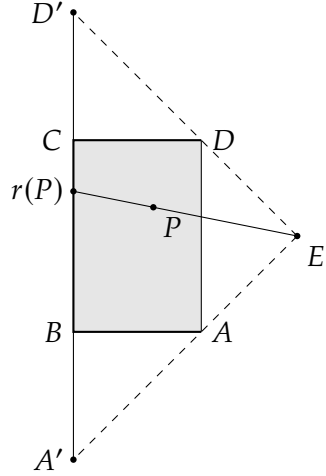
We first transfer $f$ and $h$ along congruences $AB \cong A'B$ and $CD \cong CD'$ to maps $f' : A'B \to \psquare$ and $h' : CD' \to \psquare$, and define $i : A'D' \to \psquare$ to be a common extension of $f$, $g'$ and~$h'$ by an application of~\eqref{eq:glue-abutting}.
Next, we define $r : ABCD \to A'D'$ by mapping any point~$P$ in the rectangle to the intersection of $A'D'$ and the straight line through~$E$ and~$P$.
Finally, we define $j \defeq i \circ r$.

The following lemma is an adaptation of a construction going back to~\cite{orevkov63}, see also \cite[Thm.~IV.10.1]{beeson85:_found_const_mathem}.

\begin{lemmaC}
  \label{lem:fix-point-free-map}%
  Suppose $(a_0, b_0), (a_1, b_1), (a_2, b_2), \ldots$ are open intervals with rational endpoints.
  There is a continuous map $h : ([0,1] \cap \bigcup_{i \in \NN} (a_i, b_i))^2 \to \usquare$ such that $h$ has a fixed point if, and only if, there is $n \in \NN$ for which $[0,1] \subseteq (a_0, b_0) \cup \cdots \cup (a_n, b_n)$.
\end{lemmaC}

\begin{proof}
  All intervals considered in the proof have rational endpoints, so we just call them ``intervals''.
  Throughout, we shall depend on decidability of the linear order on~$\QQ$, for example to test inclusion of one interval in another, or to tell whether a finite sequence of intervals with rational endpoints covers~$[0,1]$.

  We first consider the situation when the intervals $(a_i, b_i)$ are well-behaved in the following sense:
  \begin{itemize}
  \item no interval shares an endpoint with $(0,1)$: $a_i, b_i \not\in \set{0,1}$ for all~$i$,
  \item there are no abutting intervals: $b_i \neq a_j$ for all $i, j$, and
  \end{itemize}
  Under these circumstances $(a_i, b_i)$ and $(a_j, b_j)$ are either disjoint with a positive distance between them, or they partially overlap on an open interval.

  Define $V_k \defeq [0,1] \cap \bigcup_{i = 0}^k [a_i, b_i]$.
  We construct maps
  \begin{equation*}
    f_k : \psquare \cup V_k^2 \to \usquare
  \end{equation*}
  so that each $f_k$ extends $f_{k-1}$. In addition, we ensure that if $[0,1] \neq V_k$ then~$f_k$ does not have a fixed point and its image is contained in $\psquare$.

  Let $f_{-1} : \psquare \to \psquare$ be the rotation of~$\psquare$ by a right angle. For each $k \in \NN$, construct $f_k$ from~$f_{k-1}$ as follows:
  \begin{enumerate}
  \item
    If $V_{k-1} = V_k$ then $f_k \defeq f_{k-1}$.
  \item
    If $V_{k-1} \neq V_k \neq [0,1] $ then $V_k$ properly extends~$V_{k-1}$.
    The newly added interval $[a_k, b_k]$ may cover several gaps between prior intervals, fall entirely within such a gap, or partially reach into a gap from one of the sides. We consider the case where a single cap is covered entirely, and ask the reader treat the remaining cases analogously.

    The left-hand side of \cref{fig:fp-free} depicts a typical situation, where the light gray region is~$V_{k-1}^2$ and the dark gray one the area newly contributed in~$V_k^2$ by the addition of~$[a_k, b_k]$.
    The assumption that the intervals are well-behaved makes sure that the dark gray strips have positive widths and heights.
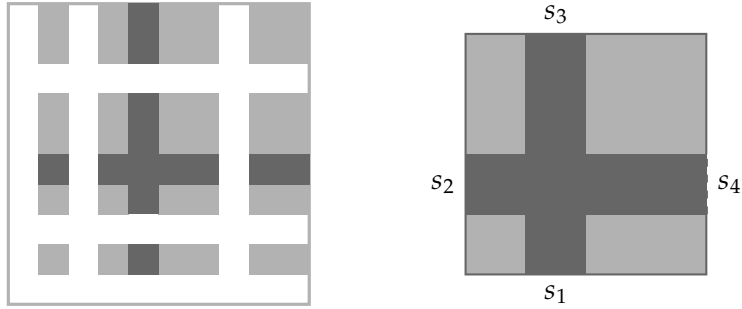
\begin{figure}[tp]
  \centering
  \begin{tikzpicture}[baseline=(current bounding box.center),scale=4]
    \draw[very thick, color=white!70!black] (0,0) rectangle +(1,1) ;
    \fill[fill=white!70!black] (0.1, 0.1) rectangle +(0.1,0.1) ;
    \fill[fill=white!70!black] (0.1, 0.3) rectangle +(0.1,0.1) ;
    \fill[fill=white!70!black] (0.1, 0.5) rectangle +(0.1,0.2) ;
    \fill[fill=white!70!black] (0.1, 0.8) rectangle +(0.1,0.2) ;
    \fill[fill=white!70!black] (0.3, 0.1) rectangle +(0.1,0.1) ;
    \fill[fill=white!70!black] (0.3, 0.3) rectangle +(0.1,0.1) ;
    \fill[fill=white!70!black] (0.3, 0.5) rectangle +(0.1,0.2) ;
    \fill[fill=white!70!black] (0.3, 0.8) rectangle +(0.1,0.2) ;
    \fill[fill=white!70!black] (0.5, 0.1) rectangle +(0.2,0.1) ;
    \fill[fill=white!70!black] (0.5, 0.3) rectangle +(0.2,0.1) ;
    \fill[fill=white!70!black] (0.5, 0.5) rectangle +(0.2,0.2) ;
    \fill[fill=white!70!black] (0.5, 0.8) rectangle +(0.2,0.2) ;
    \fill[fill=white!70!black] (0.8, 0.1) rectangle +(0.2,0.1) ;
    \fill[fill=white!70!black] (0.8, 0.3) rectangle +(0.2,0.1) ;
    \fill[fill=white!70!black] (0.8, 0.5) rectangle +(0.2,0.2) ;
    \fill[fill=white!70!black] (0.8, 0.8) rectangle +(0.2,0.2) ;
    \fill[fill=white!40!black] (0.4, 0.1) rectangle +(0.1,0.1) ;
    \fill[fill=white!40!black] (0.4, 0.3) rectangle +(0.1,0.4) ;
    \fill[fill=white!40!black] (0.4, 0.5) rectangle +(0.1,0.2) ;
    \fill[fill=white!40!black] (0.4, 0.8) rectangle +(0.1,0.2) ;
    \fill[fill=white!40!black] (0.1, 0.4) rectangle +(0.1,0.1) ;
    \fill[fill=white!40!black] (0.3, 0.4) rectangle +(0.4,0.1) ;
    \fill[fill=white!40!black] (0.5, 0.4) rectangle +(0.2,0.1) ;
    \fill[fill=white!40!black] (0.8, 0.4) rectangle +(0.2,0.1) ;
  \end{tikzpicture}
  \hfil
  \begin{tikzpicture}[baseline=(current bounding box.center),scale=8]
    \fill [fill=white!70!black] (0.3,0.3) rectangle +(0.1,0.1) ;
    \fill [fill=white!70!black] (0.3,0.5) rectangle +(0.1,0.2) ;
    \fill [fill=white!70!black] (0.5,0.5) rectangle +(0.2,0.2) ;
    \fill [fill=white!70!black] (0.5,0.3) rectangle +(0.2,0.1) ;
    \fill[fill=white!40!black] (0.4, 0.3) rectangle +(0.1,0.4) ;
    \fill[fill=white!40!black] (0.4, 0.5) rectangle +(0.1,0.2) ;
    \fill[fill=white!40!black] (0.3, 0.4) rectangle +(0.4,0.1) ;
    \fill[fill=white!40!black] (0.5, 0.4) rectangle +(0.2,0.1) ;
    \draw [thick, white!40!black, text=black]
        (0.7,0.5) -- (0.7,0.7) -- (0.5,0.7) -- (0.45,0.7) node[anchor=south] {$s_3$} -- (0.4,0.7) -- (0.3,0.7) --
        (0.3, 0.5) -- (0.3, 0.45) node[anchor=east]  {$s_2$} --
        (0.3,0.3) -- (0.45,0.3) node[anchor=north] {$s_1$} -- (0.7,0.3) -- (0.7, 0.4) ;
    \draw [very thick, white!40!black, text=black, dashed] (0.7,0.4) -- (0.7,0.45) node[anchor=west]  {$s_4$} -- (0.7,0.5) ;
  \end{tikzpicture}
  \caption{A step in the construction of~$f$ from \cref{lem:fix-point-free-map}}
  \label{fig:fp-free}
\end{figure}
    We obtain~$f_k$ by extending~$f_{k-1}$ to the dark gray area separately on each rectangular component. For example, consider the central component, shown separately on the right-hand side of the figure. Because $V_k \neq \usquare$ at least one of the line segments $s_1, s_2, s_3, s_4$ is not contained in the domain of~$f_{k-1}$, say~$s_4$.
    We extend $f_{k-1}$ to every other interval $s_i$ not contained in its domain, in our case $s_1, s_2, s_3$, by traveling along $\psquare$ from the image of one endpoint of~$s_i$ to the other.
    We then extend the map to the dark gray area by repeatedly using the technique described above: first extend to the dark gray rectangles with sides $s_1$, $s_2$, and $s_3$, then to the central square, and finally to the dark gray rectangle with the side~$s_4$.
    The reader may verify that the same approach works for other configurations.
  \item
    If $V_{k-1} \neq V_k = [0,1]$ then~$[a_k, b_k]$ fills in the last gap in~$[0,1]$.
    We visualize the situation by re-interpreting the right-hand side of the figure as showing~$\usquare$, where
    light gray is~$V_k^2$ and dark gray the newly contributed area, except that this time all four segments $s_1, s_2, s_3, s_4$ are already in the domain of~$f_{k-1}$. 
    Now we extend $f_{k-1}$ as follows:
    extend it to the four dark gray rectangles with sides $s_1, s_2, s_3, s_4$;
    then extend it to the inner dark gray square by declaring its center to be a fixed point of~$f_k$, and
    extending the rest by linear interpolation between the central point and the boundary of the square, this time using all of~$\usquare$ as the codomain of~$f_k$.
    The reader may verify that the same approach works when the dark gray area is adjacent to~$\psquare$, in which case it is shaped like the letter~L.
  \end{enumerate}
  Notice that $V_k \neq [0,1]$ implies that~$f_k$ has no fixed points. Indeed, if $t = f_k(t)$ then $t \in \psquare$, which would make~$t$ a fixed point of~$f_{-1}$.

  Let $h$ be the union of $f_k$'s, restricted to $(\usquare \cap \bigcup_{i \in \NN} (a_i, b_i))^2$.
  We must verify that~$h$ has the required property.
  If $[0,1] \subseteq (a_0, b_0) \cup \cdots \cup (a_n, b_n)$ for some $n \in \NN$, then $V_n = [0,1]$, so~$h$ has a fixed point by construction of~$f_n$.
  Conversely, if~$t$ is a fixed point of~$h$ then $t \in ([0,1] \cap (a_n, b_n))^2$ for some~$n \in \NN$, hence~$f_n$ has a fixed point, which is only possible if $V_n = [0,1]$, but then $[0,1] \subseteq (a_0, b_0) \cup \cdots \cup (a_n, b_n)$ because the intervals are well-behaved.

  It remains to remove the requirement that the intervals be well-behaved.
  Given any sequence of intervals $(a_0, b_0), (a_1, b_1), \ldots$, we define a new well-behaved sequence with the same union, which has a finite subcover of~$[0,1]$ if, and only if, $(a_0, b_0), (a_1, b_1), \ldots$ does.
  We may then apply the above construction to the new sequence.

  Let $p_i$ be the $i$-th prime, and $P_i \defeq p_1 \cdots p_i$ the product of the first~$i$ primes.
  For $i \in \NN$ and $m \in \ZZ$ let
  \begin{equation*}
    \textstyle
    c_{i,m} \defeq \frac{1 + 2 m \cdot p_i}{P_i}
    \qquad\text{and}\qquad
    d_{i,m} \defeq \frac{1 + (2 m + 3) \cdot p_i}{P_{i+1}}.
  \end{equation*}
  %
  %
  %
  %
  %
  No two intervals $(c_{i,m}, d_{i,m})$ and $(c_{j,n}, d_{j,n})$ share an endpoint, and their endpoints are all different from $0$ and $1$. Also, for a fixed~$i$ the intervals $(c_{i,m}, d_{i,m})$ form a well-behaved cover of~$\RR$.

  We enumerate some of the intervals $(c_{i,m}, d_{i,m})$ in phases, each phase contributing finitely many intervals. In the $i$-th phase we include those $(c_{i,m}, d_{i,m})$ that are contained in $(a_0, b_0) \cup \cdots \cup (a_i, b_i)$.
  This way we obtain a well-behaved sequence, as the construction of $c_{i,m}$ and $d_{i,m}$ guarantees that intervals are not abutting and that their endpoints avoid $0$ and $1$.

  Obviously, the newly enumerated intervals cover at most $\bigcup_{k \in \NN} (a_k, b_k)$. They cover all of it, because any $x \in (a_k, b_k)$ is covered at the latest by the stage at which the widths of $(c_{i,m}, d_{i,m})$'s are smaller than the distance of~$x$ to the endpoints $a_k$ and $b_k$.
  Finally, if $(a_0, b_0) \cup \cdots \cup (a_k, b_k)$ cover $[0,1]$, then they do so with a bit of overlap. There is a stage~$i$ such that the widths of $(c_{i,m}, d_{i,m})$'s are smaller than the overlap, so $[0,1]$ will be covered at least by the $i$-th stage.
\end{proof}

When the previous lemma is combined with Brouwer's fixed point theorem, a variant of Heine-Borel compactness of~$[0,1]$ emerges.

\begin{corollary}
  In the topos~$\TT{\mil}$, a countable cover of the closed unit interval by open intervals with rational endpoints has a finite subcover.
\end{corollary}

\begin{proof}
  Let $h : \objI^2 \to \objI^2$ be the map from \cref{lem:fix-point-free-map} for the given cover of~$\objI$.
  By \cref{thm:internal-brouwer} it has a fixed point, therefore \cref{lem:fix-point-free-map} ensures that~$\objI$ is covered already by a finite subcover.
\end{proof}

We can improve on the corollary to give a variant of the Drinker paradox\footnote{The ``paradox'' states that in every non-empty pub there is a person, such that if the person is drinking then everyone is drinking. It is a non-constructive principle~\cite{warren2018drinker}.} for sufficiently tame predicates on the closed unit interval.

\begin{proposition}[Drinking Buddies Principle]
  \label{prop:drinking-buddies}%
  In the topos~$\TT{\mil}$, suppose $U$ is a countable union of open intervals with rational endpoints.
  There are $x, y \in \objI$ such that $x \in U \land y \in U$ if, and only if, $\all{z \in \objI} z \in U$.
\end{proposition}

\begin{proof}
  Let $(a_0, b_0), (a_1, b_1), \ldots$ be a sequence of intervals with rational endpoints and $U \defeq \bigcup_{i \in \NN} (a_i, b_i)$ and $h : (\objI \cap U)^2 \to \objI^2$ the map from \cref{lem:fix-point-free-map} for the given intervals. We claim that $U$ is a $\neg\neg$-stable subset of~$\RR$. One way to see this is to recall from \cref{lem:lt-stable} that~$<$ is $\neg\neg$-stable, and observe that there is a map $t : \RR \to \RR$ such that $U = \set{x \in \RR \such t(x) > 0}$, for instance a weighted sum of ``igloo maps'' erected on the intervals,
  \begin{equation*}
    \textstyle
    t(x) = \sum_{n \in \NN} 2^{-n} \cdot \max (0, (a_n - x) (x - b_n) / (b_n - a_n)^2).
  \end{equation*}
  %
  %
  %
  %
  Therefore, \cref{thm:partial-brouwer} applies to~$h$ to give $(x, y) \in \objI^2$ such that if $x, y \in U$ then $h(x,y) = (x,y)$, whence $\objI \subseteq U$ by \cref{lem:fix-point-free-map}.
\end{proof}

We are not certain what the principle is good for, apart from obliging one to test its veracity with a buddy in a pub.


\subsubsection*{Acknowledgment}

We thank Ingo Blechschmidt, Joseph Miller, and Andrew Swan for valuable suggestions and engaging discussions.
This material is based upon work supported by the Air Force Office of Scientific Research under award number FA9550-21-1-0024.

\bibliographystyle{plain}
\bibliography{references}

\end{document}